\documentclass[a4paper,12pt]{article}
\usepackage{amsfonts}
\usepackage{amsmath}
\usepackage{amssymb}
\usepackage{amsthm}
\usepackage{graphicx}
\usepackage{epsf}
\usepackage{graphicx}
\usepackage{epsfig}
\usepackage{labelfig}

\usepackage{color}
\usepackage{afterpage}
\usepackage{verbatim}
\usepackage{graphicx}

\newtheorem{theorem}{Theorem}[section]
\newtheorem{proposition}[theorem]{Proposition}
\newtheorem{lemma}[theorem]{Lemma}
\newtheorem{corollary}[theorem]{Corollary}

\theoremstyle{remark}
\newtheorem{remark}[theorem]{Remark}

\numberwithin{equation}{section}

\newcommand{\abs}[1]{\lvert#1\rvert}
\newcommand{\norm}[1]{\|#1\|}
\newcommand{\Tnorm}[1]{\lvert\lvert\lvert#1\rvert\rvert\rvert}

\usepackage[small]{caption}
\begin{document}

\title{A priori error estimates and computational studies
for a  Fermi  pencil-beam equation}


\author{M. Asadzadeh$^{\ast,1}$, L. Beilina$^{\ast}$,  M. Naseer$^{\ast \ast}$ 
and C. Standar\footnote{
$^{,1}$ Corresponding author, $^*$ Department of Mathematical Sciences, Chalmers University of Technology and
Gothenburg University, SE-42196 Gothenburg, Sweden, e-mail: mohammad@chalmers.se, larisa@chalmers.se,  standarc@chalmers.se, 
$ ^{\ast \ast} $ e-mail: na.seer@hotmail.com
  } }



\date{}

\maketitle

\graphicspath{{Version_271015/}}


\begin{abstract}

We derive a priori error estimates for the standard Galerkin and 
streamline diffusion finite element methods for the Fermi pencil-beam equation 
obtained from a fully three dimensional Fokker-Planck equation 
in space ${\mathbf x}=(x,y,z)$ and velocity  
$\tilde {\mathbf v}=(\mu, \eta, \xi)$ variables. The Fokker-Planck term  
appears 
as a Laplace-Beltrami operator in the unit sphere. The diffusion term in the 
Fermi 
equation is obtained as a projection of the FP operator onto the tangent plane 
to the unit sphere at 
the pole $(1,0,0)$ and in  the direction of 
$ {\mathbf v}_0=(1,\eta, \xi)$. Hence the Fermi equation, stated 
in three dimensional spatial domain 
${\mathbf x}=(x,y,z)$, depends only on 
two velocity variables ${\mathbf v}=(\eta, \xi)$. 
Since, for a certain number of cross-sections, there is a 
closed form analytic solution available for the Fermi equation, hence an 
a posteriori error estimate procedure is unnecessary and in our 
adaptive algorithm for local mesh 
refinements we employ the a priori approach. 
Different numerical examples, in two space dimensions are 
justifying the theoretical results. Implementations show
significant reduction of the computational error by using our adaptive
algorithm.

\end{abstract}

\section{Introduction}

This work is a further development of studies in
\cite{MA97}-\cite{MA-Sopa:2002} where adaptive finite element method
was proposed for a reduction of computational cost in numerical approximation 
for pencil-beam
equations. However, focusing in theoretical convergence and stability 
aspects, except some special cases with 
limited amount of implementation in \cite{MA-EL:2007} and \cite{MA-Sopa:2002}, 
the 
detailed numerical tests were postponed to future works. 
Here, first we construct and analyze fully discrete schemes using 
both standard Galerkin and flux correcting streamline diffusion 
finite element methods for the Fermi pencil beam equation in 
three dimensions. We consider the direction of $x$-axis as the penetration  
direction of the beam particles, in the two-dimensional transverse 
spatial domain 
$\Omega_\perp:=\{x_\perp\vert x_\perp:=(y,z)\}$, moving 
with velocities in $\Omega_{\mathbf v}:\{(\eta, \xi)\}$, 
(where we have assumed $\mu\equiv 1$). We have derived our error estimates, in this geometry,
 while in numerical implementations 
we have also considered  examples in lower dimensions.

More specifically, our study concerns a ``pencil beam'' of 
neutral or charged particles 
that are normally incident on a slab of finite thickness at the spatial origin 
$(0,0,0)$ and in the direction of the positive $x$-axis. The governing equation 
for the pencil beam problem is the Fermi equation which is 
obtained by two equivalent approaches (see \cite{Borgers1996}): 
either as an asymptotic 
limit of the linear Boltzmann equation as the transport cross-section 
$\sigma_{tr}\to 0$ and the total cross-section $\sigma_t\to \infty$, or 
as an asymptotic limit in Taylor expansion of angular flux with respect  
to the velocity where the terms with 
derivatives of order three or higher are ignored. 
This procedure rely on a approach that follows the Fokker-Planck 
development. 

The Boltzmann transport equation modeling the energy independent 
pencil beam process can be written as a two-point boundary value problem viz, 
\begin{equation}\label{Boltzmann1}
\begin{split}
\mu \frac{\partial u}{\partial x}+\eta\frac{\partial u}{\partial y}+
\xi \frac{\partial u}{\partial z}=
\int_{S^2}\sigma_s({\mathbf v}\cdot {\mathbf v}^\prime)
[u({\mathbf x},{\mathbf v}^\prime)-u({\mathbf x},{\mathbf v})]\, 
d^2 {\mathbf v}^\prime, \quad 0<x<1,
\end{split}
\end{equation}
where ${\mathbf x}=(x,y,z)$ and ${\mathbf v}=(\mu, \eta, \xi)$ are 
the space and velocity vectors, respectively. The model problem concerns 
 sharply forward peaked beam of particles  
entering the spatial domain at $x=0$: 
\begin{equation}\label{Boltzmann2}
u(0, y, z,  {\mathbf v})=\delta(y)\delta(z)\frac{\delta(1-\mu)}{2\pi}, 
\qquad 0<\mu\le 1, 
\end{equation}
which are demising leaving the domain at $x=1$ (or $x=L$), viz, e.g. 
\begin{equation}\label{Boltzmann3}
u(1, y, z,  {\mathbf v})=0, 
\qquad -1\le\mu< 0. 
\end{equation}
In the realm of the Boltzmann transport equation \eqref{Boltzmann1} an overview
of the transport theory of charged particles can be found in \cite{Luo.Brahme:93}. 
In this setting a few first coefficients in a Legendre polynomial expansion 
for $\sigma_{s}$ and its integral $\sigma_t$ are parameters corresponding to 
some physical quantities of vital importance. For instance, 
the slab width in the unit of mean free path: $\sigma_t^{-1}$, 
is the reciprocal of the total cross-section 
$$
\sigma_t=2\pi\int_{-1}^1\sigma_s (\omega)\, d\omega.
$$
In the absorptionless case, the differential scattering cross section 
is given by 
\begin{equation}\label{Boltzmann4}
\sigma_s(\omega)=\sigma_t\sum_{k=0}^\infty \frac{2k+1}{4\pi}c_kP_k(\omega), 
\quad c_0=1,\quad c_1=\omega, 
\end{equation}
with $P_k(\omega)$ being the Legendre polynomial of degree $k$. 
The Fokker-Planck approximation to problem \eqref{Boltzmann1}, is based on 
using spherical harmonics expansions and yields the following, 
degenerate type partial differential equation 
\begin{equation}\label{Fokker-Planck1}
\begin{split}
\mu \frac{\partial u}{\partial x}+\eta\frac{\partial u}{\partial y}+
\xi \frac{\partial u}{\partial z}=
\frac{\sigma_{tr}}{2}\Delta_{\mathbf V}u({\mathbf x},{\mathbf v}), 
 \qquad 0<x<1,
\end{split}
\end{equation}
 associated with the same boundary data as \eqref{Boltzmann2} and 
\eqref{Boltzmann3}, and with $\Delta_{\mathbf V}$ denoting the Laplace-Beltrami 
operator 
\begin{equation}\label{Fokker-Planck2}
\Delta_{\mathbf V}:=\Big[
\frac{\partial}{\partial \mu}(1-\mu^2) \frac{\partial}{\partial \mu}+
\frac {1}{1-\mu^2}\frac{\partial^2}{\partial\phi^2}
\Big].
\end{equation}
Here, $\phi$ is the angular variable appearing in the polar representation 
$\eta=\sqrt{1-\mu^2}\cos\phi$, \qquad $\xi=\sqrt{1-\mu^2}\sin\phi$. Further, 
$\sigma_{tr}$ is the transport cross-section defined by 
$$
\sigma_{tr}=\sigma_t(1-\omega).
$$
A thorough exposition of the Fokker-Planck operator as an asymptotic limit  is given by Pomraning in 
\cite{Pomraning:92}. 
Due to successive asymptotic limits used in deriving the 
Fokker-Planck approximation, 
it is not obvious that this approximation is sufficiently 
accurate to be considered as a model for the pencil beams. 
However, for sufficiently small transport cross-section 
$\sigma_{tr}<<1$, Fermi proposed the following form of, projected, 
Fokker-Planck model: 
 \begin{equation}\label{Fermi1}
\frac{\partial u}{\partial x}+\eta\frac{\partial u}{\partial y}+
\xi \frac{\partial u}{\partial z}=\frac{\sigma_{tr}}{2}
\Big(\frac{\partial^2}{\partial\eta^2} +\frac{\partial^2}{\partial\xi^2} \Big)
u({\mathbf x},{\mathbf v}), 
 \qquad 0<x<1,
\end{equation} 
with 
\begin{equation}\label{Fermi2}
u(0,y,z,\eta,\xi)=\delta(y) \delta(z) \delta(\eta)\delta(\xi).   
\end{equation} 
Fermi's approach is different from the asymptotic ones 
and uses physical reasoning based on modeling cosmic rays. 
Note that the Fokker-Planck operator on the right hand side of 
 \eqref{Fokker-Planck1}, i.e. \eqref{Fokker-Planck2}, is the Laplacian on 
the unit sphere. The tangent plane to the unit sphere $S^2$ at 
the point $\mu_0:=(1,0,0)$ is an ${\mathcal O}(\eta^2+\xi^2)$ approximation 
to the $S^2$ at the vicinity of $\mu_0$. Extending $(\eta, \xi)$ to 
${\mathbb R}^2$, the Fourier transformation with respect to 
$y, z, \eta$, and $\xi$, assuming constant $\sigma_{tr}$, 
yields the following exact solution for the angular flux 
\begin{equation}\label{Fermi3}
u(x,y,z,\eta,\xi)=
\frac 3{\pi^2\sigma_{tr}^2x^4}\exp\Big[
-\frac 2{\sigma_{tr}}\Big(
\frac{\eta^2+\xi^2}{x}-3 \frac{y\eta+z\xi}{x^2}+
3\frac{y^2+z^2}{x^3}
\Big)
\Big] .
\end{equation} 
The closed form solution \eqref{Fermi3} was first derived by Fermi  as referred in \cite{Rossi1941}.  Eyges 
\cite{Eyges:1948} has extended this exact solution to the case of 
an $x$-depending 
$\sigma_{tr}=\sigma_{tr}(x)$. However, for the general case of 
$\sigma_{tr}({\mathbf x})=\sigma_{tr}(x,y,z)$, 
the closed-form analytic solution is not 
known. To obtain the scalar flux we integrate \eqref{Fermi3} over 
$(\eta, \xi)\in {\mathbb R}^2$: 
\begin{equation}\label{Fermi4}
\widetilde U(x,y,z)=\int_{{\mathbb R}^2}u(x,y,z,\eta,\xi)\, d\eta d\xi= 
\frac 3{2\pi\sigma_{tr}x^3}\exp\Big[-\frac 3{2\sigma_{tr}}
\Big( \frac{y^2+z^2}{x^3}\Big)
\Big]. 
\end{equation} 
Equation  \eqref{Fermi4} satisfies the transverse diffusion 
equation 
\begin{equation}\label{Fermi5}
\frac{\partial \widetilde U}{\partial x}=
\frac{\sigma_{tr} x^2}{2}\Big(\frac{\partial^2 \widetilde U}{\partial y^2}+
\frac{\partial^2\widetilde U}{\partial z^2}\Big),
\end{equation} 
with 
\begin{equation}\label{Fermi6}
\widetilde U(0,y,z)=\delta(y)\delta(z).
\end{equation}

Restricted to bounded phase-space domain, 
Fermi equation \eqref{Fermi1} can be written as the following 
``initial'' boundary value problem 
\begin{equation}\label{Fermi7}
\left\{\begin{array}{ll}
u_x+{\mathbf v}\cdot \nabla{_\perp}u=\frac{\sigma_{tr}}{2}\Delta_{\mathbf V}u \qquad & \mbox{in} \quad \Omega:=\Omega_{\mathbf x}\times  \Omega_{\mathbf v},\\
\nabla_{\mathbf v} u(x, x_\perp, {\mathbf v})=0 & \mbox{for} \quad 
(x, x_\perp, {\mathbf v})\in \Omega_{\mathbf x}\times  \partial\Omega_{\mathbf v}, \\
u(0, x_\perp, {\mathbf v})=u_0(x_\perp, {\mathbf v})  & \mbox{for} \quad 
(x_\perp, {\mathbf v})\in \Omega_{x_\perp}\times \Omega_{\mathbf v}=: \Omega_\perp, \\
u(x, x_\perp, {\mathbf v})=0 & \mbox{on} \quad \Gamma^-_{\tilde\beta}\setminus 
\{(0, x_\perp, {\mathbf v})\}, 
\end{array}
\right .
\end{equation} 
where ${\mathbf v}=(\eta, \xi)$, 
$\nabla{_\perp}=(\partial/\partial y, \partial/\partial z)$ and 
\begin{equation}\label{Fermi7A}
\Gamma^-_{\tilde\beta}:=\{(x, x_\perp,{\mathbf v})\in \partial\Omega, 
{\mathbf n}\cdot\tilde\beta <\,0   \} 
\end{equation} 
is the inflow boundary with respect to the characteristic line 
$\tilde\beta:=(1, {\mathbf v}, 0,0)$ and ${\mathbf n}$ is 
the outward unit normal 
to the boundary $\partial\Omega$. 
Note that, to derive energy estimates, 
the associated boundary data (viewed as a replacement for 
the initial data) at $x=0$ is, in a sense, approximating the product of the 
Dirac's delta functions on the right hand side of \eqref{Fermi2}.  
Assuming that we can use separation of variables, 
we may write the data function $u_0$ as product 
of two functions $f(x_\perp)$ and $g({\mathbf v})$,  
$$
u_0(x_\perp, {\mathbf v})=f(x_\perp) g({\mathbf v}).  
$$
The regularity of these functions have substantial impact in deriving 
theoretical stabilities and are essential in robustness of implemented results. 

\section{The phase-space standard Galerkin  procedure}
Below we introduce a framework that concerns a standard Galerkin 
discretization 
based on a quasi-uniform triangulation of 
the phase-space domain 
$\Omega_\perp:=\Omega_{x_\perp}\times \Omega_{\mathbf v}:=I_\perp\times \Omega_{\mathbf v}$, 
where $ I_\perp : = I_y \times I_z $. 
This is an extension of our  studies in two-dimensions in a flatland model \cite{MA-EL:2007}. 
Previous numerical approaches are mostly devoted to the study of the one-dimensional problem 
see, e.g. \cite{Larsen_etal:85} and \cite{Prinja_Pomraning:96}. 

Here we consider triangulations of the 
rectangular domains $I_\perp$ and $\Omega_{\mathbf v}:=I_\eta\times I_\xi$ 
into triangles $\tau_\perp$ and $\tau_{\mathbf v}$, and 
with the corresponding mesh parameters $h_\perp$ and $h_{\mathbf v}$, 
respectively. Then a general polynomial approximation of degree $\le r$ 
can be formulated in 
${\mathbb P}_r (\tau):
={\mathbb P}_r (\tau_\perp)\otimes {\mathbb P}_r (\tau_{\mathbf v})$. 
These polynomial spaces are more specified in the implementation section. 
We will assume a minimal angle condition on the triangles 
$ \tau_\perp $ and $ \tau_{\mathbf v} $ (see e.g.\ \cite{Brenner}).
Treating the beams entering direction $x$ 
similar to a time variable, we  
let ${\mathbf n}:={\mathbf n}(y,z, {\mathbf v})$ 
be the outward unit normal to the boundary of the phase-space domain 
$\Omega_{x_\perp}\times \Omega_{\mathbf v}$ at 
$(y,z, {\mathbf v})\in \partial\Omega_\perp$ where 
$\partial\Omega_\perp:=(\partial\Omega_{x_\perp}\times\Omega_{\mathbf v} )
\cup (\Omega_{x_\perp}\times \partial\Omega_{\mathbf v})$. 
Now set $\beta:= ({\mathbf v} ,0,0)$ and define the inflow (outflow) boundary 
as 

\begin{equation}\label{Fermi7B}
\Gamma^{-(+)}_\beta:=\{(x_\perp,{\mathbf v})\in \Gamma:= 
\partial\Omega_{\perp}:{\mathbf n}\cdot \beta 
<0 \, (>0)\}. 
\end{equation}
We shall also need an abstract finite element space 
as a subspace of a function space of Sobolev type, viz:
\begin{equation}\label{Fermi7C}
{\mathcal V}_{h,\beta}\subset 
H^1_\beta:=\{w\in H^1(I_\perp\times\Omega_{\mathbf v}): w=0\, \,\mbox{on}\,\, 
\Gamma^{-}_\beta \}.
\end{equation}
Now for all 
$w\in H^1(I_\perp\times\Omega_{\mathbf v})\cap H^r(I_\perp\times\Omega_{\mathbf v})$ 
a classical standard estimate reads as 
\begin{equation}\label{Fermi7D}
\inf_{\chi\in {\mathcal V}_{h,\beta}}
\vert\vert w-\chi\vert\vert_j\le Ch^{\alpha-j}\vert\vert w\vert\vert_\alpha, 
\quad j=1,2, \quad 1\le \alpha\le r\quad\mbox{ and } \, 
h=\max(h_\perp, h_{\mathbf v}).  
\end{equation}
To proceed let $\tilde u$ be an auxiliary interpolant of the solution $u$ 
for the equation \eqref{Fermi7} defined by 
\begin{equation}\label{Fermi7E}
{\mathcal A}( u-\tilde{u} , \chi)_\perp = 0,\,\,  \; \; \forall \; 
\chi \in {\mathcal V}_{h,\beta},
\end{equation}
where 
\begin{equation}\label{Fermi7F}
{\mathcal A}( u, w)_\perp =(u_x, w)_{\Omega_{\perp}} 
+({\mathbf v}\cdot \nabla_\perp u, w)_{\Omega_{\perp}}, 
\end{equation}
and $(\cdot, \cdot)_\perp:= (\cdot, \cdot)_{\Omega_{\perp}}=
(\cdot, \cdot)_{I_\perp\times\Omega_ { \mathbf v}}$. 

With these notation the weak formulation  for 
the problem \eqref{Fermi7} can be written as follows: 
for each $ x \in (0, L] $,
find $ u(x, \cdot ) \in H^1_{\beta}$ such that,
\begin{equation}\label{Fermi7G}
\left\{
\begin{array}{ll}
{\mathcal A}( u, \chi)_\perp+
\frac 12 (\sigma_{tr}\nabla_{\mathbf v} u, \nabla_{\mathbf v}\chi)_\perp=0 \qquad & 
\forall \chi\in H^1_{\beta}, \\
 u(0,x_\perp,{\mathbf v})=u_0(x_\perp,{\mathbf v}) &  \,\mbox{for}\,\,
(x_\perp,{\mathbf v})\in \Gamma^+_\beta, \\
 u(x,x_\perp,{\mathbf v})=0 &  \,\mbox{on}\,\, 
\Gamma^{-}_\beta\setminus\{(0,x_\perp,{\mathbf v})\}.
\end{array}
\right .
\end{equation}

Our objective is to solve the following finite element approximation 
for the problem \eqref{Fermi7G}: for each $ x \in (0, L] $,
find $u_h (x, \cdot)\in {\mathcal V}_{h,\beta}$ such that,
\begin{equation}\label{Fermi7H}
\left\{
\begin{array}{ll}
{\mathcal A}( u_h, \chi)_\perp+
\frac 12(\sigma_{tr}\nabla_{\mathbf v} u_h, \nabla_{\mathbf v}\chi)_\perp=0 \qquad & 
\forall \chi\in {\mathcal V}_{h,\beta}, \\
 u_h(0,x_\perp,{\mathbf v})=u_{0,h}(x_\perp,{\mathbf v})&  \,\mbox{for}\,\,
(x_\perp,{\mathbf v})\in \Gamma^+_\beta, \\
 u_h(x,x_\perp,{\mathbf v})=0 &  \,\mbox{on}\,\, 
\Gamma^{-}_\beta\setminus\{(0,x_\perp,{\mathbf v})\},
\end{array}
\right .
\end{equation}
where $ u_{0,h} (x_\perp, \mathbf v ) = \tilde{u} (0, x_\perp, \mathbf v ) $.

\subsection{A fully discrete scheme}
For a partition of the interval $[0, L]$ into the subintervals 
$I_{m}:=(x_{m-1}, x_m), m=1, 2, \ldots, M$ with 
$k_m:=\vert I_m\vert := x_m-x_{m-1}$, 
a finite element approximation $U$ 
with continuous linear functions $ \psi_m (x) $ on 
$I_m$ 
can be written as:
\begin{equation}\label{Fermi8}
u_h (x, x_\perp, {\mathbf v})=U_{m-1}(x_\perp, {\mathbf v})\psi_{m-1}(x)+
U_{m}(x_\perp, {\mathbf v})\psi_{m}(x), 
\end{equation}
where $x_\perp:=(y, z)$ and 
\begin{equation}\label{Fermi9}
\psi_{m-1}(x)=\frac{x_m-x}{k_m}, \qquad \psi_m (x)=\frac{x-x_{m-1}}{k_m}.  
\end{equation}
Hence, the setting \eqref{Fermi8}-\eqref{Fermi9} 
may be considered for an iterative, e.g. backward Euler,  
 scheme with continuous piecewise linear 
or discontinuous (with jump discontinuities at grid points $x_m$) 
piecewise linear functions for whole $I_x=[0,L]$. 

To proceed we consider a normalized, rectangular domain $\Omega_{\mathbf V}$ 
for the 
velocity variable 
${\mathbf v}$, as $(\eta,\xi)\in [-1,1]\times[-1,1]$ and assume a 
uniform, ``central adaptive'' discretization mesh viz: 
\begin{equation}\label{Fermi10}
\Omega_{\mathbf v}^N:=\left\{{\mathbf v}_{i,j}\subset\Omega_{\mathbf v}\Big\vert 
{\mathbf v}_{i,j}=(\eta_i, \xi_j):=\Big(\sin\frac {i \pi}{ 2n}, 
\sin\frac {j\pi}{ 2n}\Big),\,\, i,j=0, \pm 1,\ldots \pm n \right\}, 
\end{equation}
where $N=(2n+1)^2$. Further we assume that $U$ has 
compact support in $\Omega_{\mathbf V}$. By a standard approach one can 
show that, for each $m=1, 2,\ldots , M$, a finite element or finite 
difference solution $U_m^N$ obtained using the discretization 
\eqref{Fermi10} of the velocity 
domain  $\Omega_{\mathbf v}$, satisfies the $L_2(\Omega_{\mathbf v})$ error estimate 
\begin{equation}\label{Fermi11}
 \vert\vert{U_m-U_m^N}\vert\vert_{L_2(\Omega_{\mathbf v})}\le \frac C{N^2}
\vert\vert{D^2_{\mathbf V}U_m(x_\perp, \cdot)}\vert\vert_{L_2(\Omega_{\mathbf v})}. 
\end{equation}
Now we introduce a final, finite element, discretization using continuous 
piecewise linear basis functions $\varphi_j(x_\perp)$, on a partition 
${\mathcal T}_h$ of the spatial domain $\Omega_{x_\perp}$, on a quasi-uniform 
 triangulation with the mesh parameter $h$ and obtain the fully 
discrete solution $U_m^{N,h}$. We introduce discontinuities on the 
direction of entering beam (on the $x$-direction). 
We also introduce jumps appearing in passing a collision site; say  
$x_m$, as the difference between the values at 
$x_m^-$ and $x_m^+$: 
\begin{equation}\label{Fermi12}
\Big[U_m\Big]:=U_m^+-U_m^-, \qquad 
U_m^{\pm}:=\lim_{s\to 0}U(x_m\pm s, x_\perp, {\mathbf v}). 
\end{equation}
Due to the hyperbolic nature of the problem in $x_\perp$, for 
the solutions in the Sobolev space $H^{k+1}(\cdot, x_\perp,  {\mathbf v})$, 
(see Adams \cite{adams1975} for the exact definitions of the Sobolev norms and spaces) 
 the final finite element approximation yields an $L_2(\Omega_{x_\perp})$ 
error estimate viz, 
\begin{equation}\label{Fermi13}
 \vert\vert{U_m^N-U_m^{N,k}}\vert\vert_{L_2(\Omega_{x_\perp})}\le C h^{k+1/2}
\vert\vert{U_m^N(x_\perp, \cdot)}\vert\vert_{H^{k+1}(\Omega_{x_\perp})}. 
\end{equation}
To be specific, for each $m$ and each ${\mathbf v}_{i,j}\in\Omega_{\mathbf v}$ 
we obtain a spatially continuous version of the equations system \eqref{Fermi7} 
where, for $u$,  we insert  
\begin{equation}\label{Fermi14}
U_m(x_\perp, {\mathbf v}_{i,j})=
\sum_{k=1}^K U_{m,k}({\mathbf v}_{i,j})\varphi_k(x_\perp).
\end{equation}
Thus for each ${\mathbf v}_{i,j}\in\Omega_{\mathbf v}$ a variational formulation
 for a 
{\sl space-time like} discretization in $(x, x_\perp)$ of \eqref{Fermi8} 
reads as follows: find $U\in {\mathcal V}_{h,\beta}$ such that 
\begin{equation}\label{Fermi14}
\begin{split}
\int_{I_m}\int_{\Omega_{x_\perp}} & U_x(x,x_\perp, {\mathbf v}_{i,j})\varphi_k(x_\perp) dx_\perp dx
+\int_{I_m}\int_{\Omega_{x_\perp}}
{\mathbf v}_{i,j}\cdot\nabla_\perp U(x,x_\perp, {\mathbf v}_{i,j})\varphi_k(x_\perp) dx_\perp dx\\
&= \int_{I_m}\int_{\Omega_{x_\perp}} \frac{\sigma_{tr}}2\Delta_ {\mathbf V}
U(x,x_\perp, {\mathbf v}_{i,j})\varphi_k(x_\perp) dx_\perp dx, \qquad \forall \,\, 
\varphi_k\in {\mathcal V}_{h,\beta}, 
\end{split} 
\end{equation}
where 
\begin{equation}\label{Fermi14A}
{\mathcal V}_{h,\beta}:=
\{w\in {\mathcal V}_{\beta}\vert w_{ \lvert_\tau } \in \mathbb{P}_1 ( \tau ) , w \,\, \mbox{is continuous} \}.
\end{equation}
This yields 
\begin{equation}\label{Fermi15}
\begin{split}
\int_{\Omega_{x_\perp}} & 
\Big(U_m(x_\perp, {\mathbf V}_{i,j})-U_{m-1}(x_\perp, {\mathbf v}_{i,j})\Big)
\varphi_k(x_\perp) dx_\perp + 
\frac{k_m}2 \int_{\Omega_{x_\perp}}{\mathbf v}_{i,j}\cdot\nabla_\perp 
U_m(x_\perp, {\mathbf V}_{i,j})\varphi_k(x_\perp) dx_\perp \\
&+
\frac{k_m}2 \int_{\Omega_{x_\perp}}{\mathbf v}_{i,j}\cdot\nabla_\perp 
U_{m-1}(x_\perp, {\mathbf v}_{i,j})\varphi_k(x_\perp) dx_\perp \\
&=\frac{\sigma_{tr}}2\frac{k_m}2 \int_{\Omega_{x_\perp}}\Delta_ {\mathbf v}
U_m(x_\perp, {\mathbf v}_{i,j})\varphi_k(x_\perp) dx_\perp +
\frac{\sigma_{tr}}2\frac{k_m}2 \int_{\Omega_{x_\perp}}\Delta_ {\mathbf v}
U_{m-1}(x_\perp, {\mathbf v}_{i,j})\varphi_k(x_\perp) dx_\perp .
\end{split} 
\end{equation}
Such an equation would lead to a linear system of equations which in 
compact form can be written as the following matrix equation 
\begin{equation}\label{Fermi16}
\begin{split}
MU_m({\mathbf v}_{i,j})& -MU_{m-1}({\mathbf v}_{i,j})  
+\frac{k_m}2 C_{{\mathbf v}_{i,j}} U_m({\mathbf v}_{i,j})+
\frac{k_m}2 C_{{\mathbf v}_{i,j}} U_{m-1}({\mathbf v}_{i,j})\\
&=\frac{\sigma_{tr}}2\frac{k_m}2 \Delta_ {\mathbf v}MU_m({\mathbf v}_{i,j})+
\frac{\sigma_{tr}}2\frac{k_m}2 \Delta_ {\mathbf v}MU_{m-1}({\mathbf v}_{i,j})
\end{split} 
\end{equation}
Now considering ${\mathbf v}$-continuous version of 
\eqref{Fermi16}: 
\begin{equation}\label{Fermi17}
\begin{split}
MU_m({\mathbf v})& -MU_{m-1}({\mathbf v})  
+\frac{k_m}2 C_{{\mathbf V}} U_m({\mathbf v})+
\frac{k_m}2 C_{{\mathbf V}} U_{m-1}({\mathbf v})\\
&=\frac{\sigma_{tr}}2\frac{k_m}2 \Delta_ {\mathbf V}MU_m({\mathbf v})+
\frac{\sigma_{tr}}2\frac{k_m}2 \Delta_ {\mathbf V}MU_{m-1}({\mathbf v}), 
\end{split} 
\end{equation}
we may write 
\begin{equation}\label{Fermi18}
U_m(x_\perp, {\mathbf v})=
\sum_{k=1}^K \sum_{j=1}^J U_{m,k, j}\chi_{j}({\mathbf v})\varphi_k(x_\perp).
\end{equation}
Then a further variational form is obtained by multiplying 
\eqref{Fermi17} by $\chi_{j}$, $j=1,2, \ldots, J$ and integrating over 
$\Omega_{\mathbf V}$: 
\begin{equation}\label{Fermi19}
\begin{split}
&\Big(M_{x_\perp}\otimes M_{\mathbf v}\Big) U_m 
+\frac{k_m}2 \Big(C_{x_\perp}\otimes\widetilde M_{\mathbf v} \Big) U_m
+\frac{\sigma_{tr}}2\frac{k_m}2 \Big(S_{\mathbf v}\otimes M_{x_\perp}\Big)U_m\\
&\qquad =\Big(M_{x_\perp}\otimes M_{\mathbf v} \Big)U_{m-1}-
\frac{k_m}2 \Big(C_{x_\perp}\otimes\widetilde M_{\mathbf v} \Big)U_{m-1}
-\frac{\sigma_{tr}}2\frac{k_m}2 \Big(S_ {\mathbf v}\otimes M_{x_\perp}\Big)U_{m-1}, 
\end{split} 
\end{equation}
where $\otimes$ represents tensor products 
with the obvious notations for the coefficient matrices 
$M_{x_\perp}$, $M_{\mathbf v}$ being the mass-matrices in spatial and velocity 
variables, $C_{x_\perp}$ is the convection matrix in space, 
$\widetilde M_{\mathbf v}:={\mathbf v}\otimes M_{\mathbf v}$ corresponds to the spatial 
convection terms with the coefficient 
${\mathbf v}$: ${\mathbf v}\cdot \nabla_\perp$, and finally 
$S_ {\mathbf v}$ is the stiffness matrix in  ${\mathbf v}$. 

Now, given an initial beam configuration, $U_0=u_0$, our objective is to 
use an iteration algorithm as the finite element version above  
or the corresponding equivalent backward Euler (or Crank-Nicolson) approach 
for discretization in the $x_\perp$ variable, and obtain successive $U_m$-values 
at the subsequent discrete $x_\perp$-levels. To this end the delicate issues of 
an initial data viz \eqref{Fermi2}, as a product of Dirac delta functions, 
as well as the desired dose to the target  that imposes
 the model to be transferred 
to a case having an {\sl inverse problem} nature are challenging 
practicalities.

\subsubsection{Standard stability estimates} 
We use the notion of the scalar products over a domain 
${\mathbf D}$ and its boundary $\partial {\mathbf D}$ as 
$(\cdot, \cdot)_{\mathbf D}$ and 
$\langle \cdot, \cdot \rangle_{\partial {\mathbf D}}$,  
respectively. Here, 
${\mathbf D}$ can be $\Omega:=I_x\times\Omega_{\mathbf x}\times\Omega_{\mathbf v}$,  
$ I_x\times\Omega_{\mathbf x}$, $\Omega_{\mathbf x}\times\Omega_{\mathbf v}$, or 
possibly other relevant domains in the problem. 
Below we state and prove a stability lemma which, in some adequate norms, 
 guarantees the control of the 
solution for the continuous problem by the data. 
The lemma is easily 
extended to the case of approximate solution. 

We derive the stability estimate using the triple norm 
\begin{equation}\label{Fermi19A}
\vert \vert\vert w\vert\vert\vert_\beta^2 = 
 \int_0^L \int_{\Gamma_\beta^+} w^2({\mathbf n}\cdot \beta)\, d\Gamma dx
+\vert\vert \sigma_{tr}^{1/2}\nabla_{\mathbf v} w\vert\vert^2_{L_2 (\Omega)}. 
\end{equation}

\begin{lemma} \label{Lemma1}
For $u$ satisfying \eqref{Fermi7} we have the stability estimates 
\begin{equation}\label{Fermi19A1}
\sup_{x\in I_x}\norm{u(x,\cdot, \cdot)}_{L_2(\Omega_\perp\times\Omega_{\mathbf v})}
\le \norm{u_0(\cdot, \cdot)}_{L_2(\Omega_\perp\times\Omega_{\mathbf v})},
\end{equation}
and 
\begin{equation}\label{Fermi19A2}
\vert \vert\vert u\vert\vert\vert_\beta \le 
\norm{u_0(\cdot, \cdot)}_{L_2(\Omega_\perp\times\Omega_{\mathbf v})}. 
\end{equation}

\end{lemma} 
\begin{proof} 
We let $\chi=u$ in \eqref{Fermi7G} and use \eqref{Fermi7E}-\eqref{Fermi7F} 
to obtain 
\begin{equation}\label{Fermi19B}
\frac 12 \frac {d}{dx}\norm{u}_{L_2(I_\perp\times\Omega_{\mathbf v})}^2 +
({\mathbf v}\cdot\nabla_\perp u, u)_{I_\perp\times\Omega_{\mathbf v} } +
\frac 12 \norm{\sigma_{tr}^{1/2}\nabla_{\mathbf v}u}_{L_2(I_\perp\times\Omega_{\mathbf v})}^2 = 0, 
\end{equation}
where using Green's formula and with $\beta=({\mathbf v},0)$ we have 
\begin{equation}\label{Fermi19B1}
\begin{split}
({\mathbf v}\cdot\nabla_\perp u, u)_{I_\perp\times\Omega_{\mathbf v} }& =
\int_{\Omega_\mathbf v}\Big(\int_{I_\perp} 
({\mathbf v}\cdot\nabla_\perp u)u\, \Big) dx_\perp  \, d{\mathbf v}\\
&=
\frac 12 \int_{\Omega_{\mathbf v}} ({\mathbf n}\cdot{\mathbf v}) u^2\, d{\mathbf v}
=\frac 12 \int_{\Gamma_\beta^+} u^2({\mathbf n}\cdot \beta)\, d\Gamma\ge 0.
\end{split}
\end{equation}
Thus 
\begin{equation}\label{Fermi19B2}
\frac {d}{dx}\norm{u}_{L_2(I_\perp\times\Omega_{\mathbf v})}^2 \le 0,
\end{equation}
which yields \eqref{Fermi19A1} after integration over $ (0, x)$ and taking supremum over $ x \in I_x $. 
Integrating \eqref{Fermi19B} over $x\in (0, L)$ and using \eqref{Fermi19B1}
together with the definition of the triple norm $ ||| \cdot |||_\beta $ we get 
\begin{equation}\label{Fermi19B3}
\norm{u(L,\cdot,\cdot)}_{L_2(I_\perp\times\Omega_{\mathbf v})}^2
+ ||| u |||_{\beta}^2
=\norm{u(0,\cdot,\cdot)}_{L_2(I_\perp\times\Omega_{\mathbf v})}^2 
\end{equation}
and the estimate \eqref{Fermi19A2} follows.
\end{proof} 

Using the same argument as above we obtain the semi-discrete version of the 
stability Lemma \ref{Lemma1}:
\begin{corollary}
The semi-discrete 
solution $u_h$ with $h=\max (h_\perp, h_{\mathbf v})$ and standard Galerkin 
approximation in {\sl phase-space} $I_\perp\times\Omega_{\mathbf v}$ 
satisfies the semidiscrete stability estimates: 
\begin{equation}\label{Fermi19B4}
\sup_{x\in I_x}\norm{u_h(x,\cdot, \cdot)}_{L_2(\Omega_\perp\times\Omega_{\mathbf v})}
\le \norm{u_{0,h}(\cdot, \cdot)}_{L_2(\Omega_\perp\times\Omega_{\mathbf v})},
\end{equation}
\begin{equation}\label{Fermi19B5}
\vert \vert\vert u_h\vert\vert\vert_\beta \leq 
\norm{u_{0,h}(\cdot, \cdot)}_{L_2(\Omega_\perp\times\Omega_{\mathbf v})}. 
\end{equation}

\end{corollary}

\subsubsection{Convergence} 

Below we state and prove an a priori error estimate for the finite element
approximation $ u_h $ satisfying \eqref{Fermi7H}. The a priori error estimate
will be stated in the triple norm defined by \eqref{Fermi19A}.

\begin{lemma}\label{aprioriLemma} 
[An a priori error estimate in the triple norm] 
Assume that $u$ and $u_h$ satisfy the continuous and discrete problems 
\eqref{Fermi7G} and \eqref{Fermi7H}, respectively. Let 
$u\in H^r(\Omega)= H^r( \Omega_{\mathbf x}\times \Omega_{\mathbf v})$, 
$r\ge 2$, then there is a 
constant $C$ independent of $ v$, $u$ and $h$ such that 
\begin{equation}\label{Fermi20A}
\Tnorm{u-u_h}_{\tilde\beta}\le Ch^{r-1/2}\norm{u}_{H^r(\Omega)}.
\end{equation}
\end{lemma}

\begin{proof}
 Taking the first equations in \eqref{Fermi7G} and \eqref{Fermi7H} and using 
\eqref{Fermi7F} we end up with 
\begin{equation}\label{Fermi20A1}
\begin{split}
\Big((u_h-\tilde u)_x, \chi\Big)_{I_\perp\times\Omega_{\mathbf v}} &+ 
\Big({\mathbf v}\cdot 
\nabla_{\perp}(u_h-\tilde u), \chi\Big)_{I_\perp\times\Omega_{\mathbf v}} 
+\frac 12 \Big(\sigma_{tr}\nabla_{\mathbf v}
(u_h-\tilde u), \nabla_{\mathbf v}\chi\Big)_{I_\perp\times\Omega_{\mathbf v}}\\
& =\frac 12\Big(\sigma_{tr} \nabla_{\mathbf v}(u-\tilde u), \nabla_{\mathbf v}\chi\Big)_{I_\perp\times\Omega_{\mathbf v}}. 
\end{split}
\end{equation}
Let now $\chi=u_h-\tilde u$, then by the same argument as in the proof of 
Lemma \ref{Lemma1} we get 
\begin{equation}\label{Fermi20A2}
\begin{split}
& \frac {d}{dx}
\norm{u_h-\tilde u}^2_{L_2(I_\perp\times\Omega_{\mathbf v})}+
\int_{\Gamma_\beta^+}({\mathbf n}\cdot\beta) (u_h-\tilde u)^2\, d\Gamma+
\norm{\sigma_{tr}^{1/2}
\nabla_{\mathbf v}(u_h-\tilde u)}^2_{L_2(I_\perp\times\Omega_{\mathbf v})}\\
&\quad\le \frac 12\norm{\sigma_{tr}^{1/2}
\nabla_{\mathbf v}(u_h-\tilde u)}^2_{L_2(I_\perp\times\Omega_{\mathbf v})}
+\frac 12\norm{\sigma_{tr}^{1/2}
\nabla_{\mathbf v}(u-\tilde u)}^2_{L_2(I_\perp\times\Omega_{\mathbf v})}, 
\end{split}
\end{equation}
which yields 
\begin{equation}\label{Fermi20A3}
\begin{split}
 \frac {d}{dx}
\norm{u_h-\tilde u}^2_{L_2(I_\perp\times\Omega_{\mathbf v})} &+
\int_{\Gamma_\beta^+}({\mathbf n}\cdot\beta) (u_h-\tilde u)^2\, d\Gamma+
\frac 12\norm{\sigma_{tr}^{1/2}
\nabla_{\mathbf v}(u_h-\tilde u)}^2_{L_2(I_\perp\times\Omega_{\mathbf v})}\\
&\le 
\frac 12 \norm{\sigma_{tr}^{1/2}
\nabla_{\mathbf v}(u-\tilde u)}^2_{L_2(I_\perp\times\Omega_{\mathbf v})}.  
\end{split}
\end{equation}
Hence, integrating over $x\in (0,L)$ we obtain 
\begin{equation}\label{Fermi20A4}
\begin{split}
&\norm{(u_h-\tilde u)(L,\cdot,\cdot)}^2_{L_2(I_\perp\times\Omega_{\mathbf v})} + 
\int_{\Gamma_\beta^+\setminus \Gamma_L}
(\tilde{\mathbf n}\cdot\tilde\beta) (u_h-\tilde u)^2\, d\Gamma + 
\frac 12\norm{\sigma_{tr}^{1/2} 
\nabla_{\mathbf v}(u_h-\tilde u)}^2_{L_2(I_x\times I_\perp\times\Omega_{\mathbf v})}\\
&\qquad  \le 
\frac 12\norm{\sigma_{tr}^{1/2}
\nabla_{\mathbf v}(u-\tilde u)}^2_{L_2(I_x\times I_\perp\times\Omega_{\mathbf v})}+
 \norm{(u_h-\tilde u)(0,\cdot,\cdot)}^2_{L_2(I_\perp\times\Omega_{\mathbf v})}, 
\end{split}
\end{equation}
where 
$\Gamma_\beta^+\setminus \Gamma_L:=\{\{L\}\times I_\perp\times\Omega_{\mathbf v}\}$. 
Now recalling that $u_h(0,\cdot,\cdot)=\tilde u(0,\cdot,\cdot)=u_{0,h}$ 
and  the definition of $\Tnorm{\cdot}_{\tilde\beta}$ we end up with 
\begin{equation}\label{Fermi20A5}
\Tnorm{u_h-\tilde u}_{\tilde\beta}\le 
\norm{\sigma_{tr}^{1/2}
\nabla_{\mathbf v}(u-\tilde u)}^2_{L_2(I_x\times I_\perp\times\Omega_{\mathbf v})}.
\end{equation}
Finally using the identity 
$u_h-u=(u_h-\tilde u)+ (\tilde u-u)$ and the interpolation estimate below 
we obtain the desired result. 
\end{proof}

\begin{proposition}[See \cite{Ciarlet:80}] 
Let $h^2\le \sigma_{tr}({\mathbf x})\le h$, 
then there is an interpolation constant 
$\tilde C$ such that 
\begin{equation}\label{Fermi20A6}
\Tnorm{u-\tilde u}_{\tilde\beta}\le \tilde C h^{r-1/2}\norm{u}_r.
 \end{equation}
\end{proposition}

\begin{proof}
We rely on classical interpolation error estimates 
(see \cite{Brenner} and \cite{Ciarlet:80}): Let $u\in H^r(\Omega)$, then 
there exists interpolation constants $C_1$ and $C_2$ such that for the nodal 
interpolant $\pi_h\in {\mathcal V}_{h, \tilde\beta}$ of $u$ 
we have the interpolation error estimates 
\begin{align}
\norm{u-\pi_hu}_s & \le C_1h^{r-s}\norm{u}_r,\qquad s=0,\,\,1 \label{InterpolI}\\
\abs{u-\pi_hu}_{\tilde\beta} & \le C_2h^{r-1/2}\norm{u}_r,\label{InterpolII}
\end{align}
where 
$$
\abs{\varphi}_{\tilde\beta}:=
\Big(\int_{\Gamma_{\tilde\beta}}\varphi^2 ({\mathbf n}\cdot \tilde\beta)\, d\Gamma 
\Big)^{1/2}.
$$
Using the definition of the triple-norm we have that 
\begin{equation}\label{InterpolIA}
\begin{split}
\Tnorm{u-\pi_hu}_{\tilde\beta}^2 = &
 \abs{u-\pi_hu}_{\tilde\beta}^2+
\norm{\sigma_{tr}^{1/2}\nabla_{\mathbf v}(u-\pi_hu)}^2\\
&\le  \abs{u-\pi_hu}_{\tilde\beta}^2
+\norm{\sigma_{tr}^{1/2}}_\infty^2\norm{u-\pi_hu}_{H^1(\Omega)}^2\\
&\le C_2^2 h^{2r-1}\norm{u}_r^2+
C_1^2\sup_{{\mathbf x}}\abs{\sigma_{tr}} h^{2r-1}\norm{u}_r^2\\
=& 
\Big( C_2^2 +C_1^2\sup_{{\mathbf x}}\abs{\sigma_{tr}}\Big) 
h^{2r-1}\norm{u}_r^2
\end{split} 
\end{equation}
where  
in the last inequality we have used \eqref{InterpolI} and 
\eqref{InterpolII}. Now choosing the constant 

$\tilde C=\Big( C_2^2 +C_1^2\sup_{{\mathbf x}}\abs{\sigma_{tr}}\Big)^{1/2}$ 
we get the desired result. 
\end{proof}

This proposition yields the $L_2$ error estimate viz:

\begin{theorem}[$L_2$ error estimate]
For $u\in H^r(\Omega)$ and $u_h\in {\mathcal V}_{h,\tilde\beta}$ satisfying 
\eqref{Fermi7G} and \eqref{Fermi7H}, respectively, and with 
$h^2\le \sigma_{tr}\le h$, 
 we have that there is a constant 
$C=C(\Omega, f)$ such that 
\begin{equation}\label{L2Estimate}
\norm{u-u_h}_{L_2(\Omega)}\le C h^{r-3/2} 
\norm{u}_r. 
\end{equation}
\end{theorem}
\begin{proof}
Using the Poincar\'e inequality 
\begin{equation}\label{L2EstimateA}
\norm{u-u_h}_{L_2(\Omega)}\le C \norm{\nabla_{\mathbf v}( u-u_h)}_{L_2(\Omega)}
\le \frac C{\min \sigma_{tr}^{1/2}}\norm{ \sigma_{tr}^{1/2}\nabla_{\mathbf v}( u-u_h)}_{L_2(\Omega)}
\end{equation}
Further using Lemma \ref{aprioriLemma}  
\begin{equation}\label{L2EstimateB}
\norm{\sigma_{tr}^{1/2}\nabla_{\mathbf v}(u-u_h)}\le \Tnorm{u-u_h}_{\tilde\beta}
\le C h^{r-1/2} \norm{u}_r. 
\end{equation}
Combining \eqref{L2EstimateA}, \eqref{L2EstimateB} 
and recalling that $\sigma_{tr}$ is in the interval $[h^2, h]$ we end up with 
\begin{equation}\label{L2EstimateC}
\norm{u-u_h}_{L_2(\Omega)}\le C h^{r-3/2} \norm{u}_r, 
\end{equation}
and the proof is complete. 
\end{proof}

\section{Petrov-Galerkin approaches} 

Roughly speaking, in the Petrov-Galerkin method one adds a streaming term to 
the test function.
The raison d\'e etre of such approach is described, 
motivated and analyzed in the classical SD methods.
Here, our objective is to briefly introduce a few cases of 
Petrov-Galerkin approaches in some
lower dimensional geometry and implement them in both 
direct and adaptive settings. 
Some specific form of the Petrov-Galerkin methods are studied in \cite{Johnson1992} where the 
method of exact transport + projection is introduced. Also both the semi-streamline diffusion as
well as the Characteristic streamline diffusion methods, 
which in their simpler forms are implemented here, are studied in \cite{MA2002}. 

\subsection{A semi-streamline diffusion scheme (SSD)}
Here the main difference with the standard approach is that we 
employ modified test functions of the form 
$w+\delta {\mathbf v}\cdot\nabla_\perp w$  with 
$\delta \ge \sigma_{tr}$. Further, we assume that 
$w$ satisfies the vanishing inflow boundary condition 
of \eqref{Fermi7}. Hence, multiplying the differential equation 
in \eqref{Fermi7} by $w+\delta({\mathbf v}\cdot\nabla_\perp w)$ and 
integrating over $\Omega_\perp=\Omega_{{\mathbf x}_\perp}\times\Omega_{\mathbf v}$ 
we have a variational formulation, viz 
\begin{equation}\label{SSD1}
\begin{split}
\Big(u_x & +{\mathbf v}\cdot \nabla_\perp u
-\frac 12\sigma_{tr}\Delta_{\mathbf v}u , 
w+\delta({\mathbf v}\cdot\nabla_\perp w)\Big)_\perp
=(u_x, w)_\perp+\delta(u_x, {\mathbf v}\cdot\nabla_\perp w)_\perp+
({\mathbf v}\cdot \nabla_\perp u, w)_\perp \\
&+ 
\delta( {\mathbf v}\cdot \nabla_\perp u,{\mathbf v}\cdot\nabla_\perp w)_\perp
+\frac 12(\sigma_{tr} \nabla_{\mathbf v} u,  \nabla_{\mathbf v}w)_\perp
+\frac {\delta}{2}( \sigma_{tr} \nabla_{\mathbf v} u, 
\nabla_{\mathbf v}({\mathbf v}\cdot\nabla_\perp w))_\perp=0.
\end{split}
\end{equation}

\subsubsection{The SSD stability estimate}
We let in \eqref{SSD1} $w=u$ and obtain the following identity
\begin{equation}\label{SSD1A}
\begin{split}
\frac 12 \frac {d}{dx} \norm{u}_\perp^2 & + 
\delta (u_x, {\mathbf v}\cdot \nabla_\perp u)_\perp+
\frac 12 \int_{\Gamma_\beta^+}({\mathbf n}\cdot\beta) u^2\, d\Gamma+ 
\delta\norm{{\mathbf v}\cdot \nabla_\perp u}_\perp^2 \\
&+\frac 12\norm{\sigma_{tr}^{1/2} \nabla_{\mathbf v} u}_\perp^2
+\frac {\delta}{2} \Big(\sigma_{tr} \nabla_{\mathbf v} u, 
\nabla_{\mathbf v}({\mathbf v}\cdot\nabla_\perp u)\Big)_\perp=0. 
\end{split}
\end{equation}
Now it is easy to verify that the last term above can be written as 
\begin{equation}\label{SSD1B}
\delta\Big(\sigma_{tr} \nabla_{\mathbf v} u, 
\nabla_{\mathbf v}({\mathbf v}\cdot\nabla_\perp u)\Big)_\perp
= \delta \int_{\Omega_\perp}\sigma_{tr}\Big(\nabla_{\mathbf v}u\cdot \nabla_\perp u
+\frac 12 {\mathbf v}\cdot \nabla_\perp(\abs{\nabla_{\mathbf v} u})^2\Big)\, 
dx_\perp \, d{\mathbf v}. 
\end{equation}
Due to symmetry the second term in the integral above vanishes. 
Hence we end up with 
\begin{equation}\label{SSD1C}
\begin{split}
\frac 12 \frac {d}{dx} \norm{u}_\perp^2 & + 
\delta (u_x, {\mathbf v}\cdot \nabla_\perp u)_\perp+
\frac 12 \int_{\Gamma_\beta^+}({\mathbf n}\cdot\beta) u^2\, d\Gamma+ 
\delta\norm{{\mathbf v}\cdot \nabla_\perp u}_\perp^2 \\
&+\frac 12\norm{\sigma_{tr}^{1/2} \nabla_{\mathbf v} u}_\perp^2
+\frac {\delta}{2} \Big(\sigma_{tr} \nabla_{\mathbf v} u, \nabla_\perp u)\Big)_\perp=0. 
\end{split}
\end{equation}
Next, we multiply the differential equation \eqref{Fermi7} by 
$\delta u_x$ and integrate over $I_\perp\times\Omega_{{\mathbf v}}$ to get 
\begin{equation}\label{SSD1D}
\delta\norm{u_x}^2+(\delta u_x, {\mathbf v}\cdot \nabla_\perp u)_\perp
+\frac {\delta}{2} (\sigma_{tr} \nabla_{{\mathbf v}}u, \nabla_{{\mathbf v}}u_x)_\perp=0.  
\end{equation}
The last inner product on the left hand side of \eqref{SSD1D} can be written as 
\begin{equation}\label{SSD1E}
(\sigma_{tr} \nabla_{{\mathbf v}}u, \nabla_{{\mathbf v}}u_x)_\perp=
\frac 12\frac {d}{dx}\int_{I_\perp\times\Omega_{{\mathbf v}}}\sigma_{tr} 
\abs{\nabla_{{\mathbf v}}}^2\, dx_\perp\, d {\mathbf v}-
\frac 12\int_{I_\perp\times\Omega_{\mathbf v}}
\frac{\partial\sigma_{tr}}{\partial x}\Big(\abs{\nabla_{\mathbf v}}^2\Big)
\, dx_\perp\, d {\mathbf v}.
\end{equation}
Now inserting \eqref{SSD1E} in \eqref{SSD1D} and adding the 
result to \eqref{SSD1C} 
we end up with 
\begin{equation}\label{SSD1F}
\begin{split}
\frac 12 \frac {d}{dx} \norm{u}_\perp^2 & + 
\delta\norm{u_x+ {\mathbf v}\cdot \nabla_\perp u}_\perp^2+ 
\frac 12 \int_{\Gamma_\beta^+}({\mathbf n}\cdot\beta) u^2\, d\Gamma 
+\frac 12\norm{\sigma_{tr}^{1/2} \nabla_{\mathbf v} u}_\perp^2\\
&+\frac {\delta}{2} \Big(\sigma_{tr} \nabla_{\mathbf v} u, \nabla_\perp u \Big)_\perp
+\frac {\delta}4\frac {d}{dx}\int_{I_\perp\times\Omega_{{\mathbf v}}}\sigma_{tr} 
\abs{\nabla_{{\mathbf v}}}^2\, dx_\perp\, d {\mathbf v}\\
&-
\frac {\delta}4\int_{I_\perp\times\Omega_{\mathbf v}}
\frac{\partial\sigma_{tr}}{\partial x}\Big(\abs{\nabla_{\mathbf v}}^2\Big)
\, dx_\perp\, d {\mathbf v}=0.
\end{split}
\end{equation}
Further we use the Cauchy-Schwarz inequality to get 
\begin{equation}\label{SSD1G}
\abs{\Big(\sigma_{tr} \nabla_{\mathbf v} u, \nabla_\perp u\Big)_\perp}
\le \frac 12\norm{\sigma_{tr}^{1/2}\nabla_\perp u}_\perp^2+ 
\frac 12\norm{\sigma_{tr}^{1/2}\nabla_{\mathbf v} u}_\perp^2. 
\end{equation}
Finally with an additional symmetry assumption on 
$x_\perp$ and ${\mathbf v}$ convections as 
(this is motivated by forward peakedness assumption in angle and energy 
which is used in deriving 
the Fokker-Plank/Fermi equations) 
\begin{equation}\label{SSD1H}
\norm{\sigma_{tr}^{1/2}\nabla_\perp u}_\perp\sim 
\norm{\sigma_{tr}^{1/2}\nabla_{\mathbf v} u}_\perp, 
\end{equation}
and the fact that $\sigma_{tr}$ is decreasing in the beams penetration 
direction, i.e. $\frac{\partial\sigma_{tr}}{\partial x}\le 0$, 
we may write \eqref{SSD1F} as 
\begin{equation}\label{SSD1I}
\begin{split}
\frac 12\frac {d}{dx}\Big(\norm{u}_\perp^2 +\frac 12 \delta 
\int_{I_\perp\times\Omega_{{\mathbf v}}}\sigma_{tr} 
\abs{\nabla_{{\mathbf v}}}^2\, dx_\perp\, d {\mathbf v}\Big) & +
\frac 12 \int_{\Gamma_\beta^+}({\mathbf n}\cdot\beta) u^2\, d\Gamma 
+\delta\norm{u_x+ {\mathbf v}\cdot \nabla_\perp u}_\perp^2\\
&+\frac 12
(1-\delta)\norm{\sigma_{tr}^{1/2}\nabla_{\mathbf v} u}_\perp^2\le 0.
\end{split}
\end{equation}
As a consequence for sufficiently small $\delta$ 
(actually $\delta\approx \sigma_{tr}^{1/2}\ll 1$) we have, e.g. 
\begin{equation}\label{SSD1J}
\frac {d}{dx}\Big(\norm{u}_\perp^2 + \frac 12\delta 
\int_{I_\perp\times\Omega_{{\mathbf v}}}\sigma_{tr} 
\abs{\nabla_{{\mathbf v}}}^2\, dx_\perp\, d {\mathbf v}\Big) <0 ,
\end{equation}
and hence $\norm{u}_\perp^2 +\frac {\delta}{2} 
\int_{I_\perp\times\Omega_{{\mathbf v}}}\sigma_{tr} 
\abs{\nabla_{{\mathbf v}}}^2\, dx_\perp\, d {\mathbf v}$ 
is strictly decreasing in $x$. Consequently, for each $x^\prime\in [0,L]$ 
we have that 
\begin{equation}\label{SSD1K}
\norm{u(x^\prime, \cdot,\cdot)}^2_{\perp}+
\frac {\delta}{2} \norm{\sigma_{tr}^{1/2}\nabla_{\mathbf v} u(x^\prime, \cdot,\cdot)}^2_\perp
\le 
\norm{u(0, \cdot,\cdot)}^2_{L_2(I_\perp\times\Omega_{\mathbf v})}+
\frac {\delta}{2} \norm{\sigma_{tr}^{1/2}\nabla_{\mathbf v} u(0, \cdot,\cdot)}^2_\perp.
\end{equation}
Thus, summing up we have proved the following stability estimates 
\begin{proposition}\label{SSDstabilities} 
Under the assumption \eqref{SSD1H} the following 
$L_2(I_\perp\times \Omega_{\mathbf v})$ stability holds true 
\begin{equation}\label{SSD1L}
\norm{u(L, \cdot,\cdot)}^2_{\perp}+
\frac {\delta}{2} \norm{\sigma_{tr}^{1/2}\nabla_{\mathbf v} u(L, \cdot,\cdot)}^2_\perp
\le \norm{u_0}_{\perp}^2
+\frac {\delta}{2} \norm{\sigma_{tr}^{1/2}\nabla_{\mathbf v} u_0}_{\perp}^2. 
\end{equation}
Moreover, 
we have also the second stability estimate 
\begin{equation}\label{SSD1M}
\Tnorm{u}_{\tilde\beta}^2+
\delta\norm{u_x+ {\mathbf v}\cdot \nabla_\perp u}_{L_2(\Omega)}^2
\le \tilde C\Big(\norm{u_0}_{L_2(I_\perp\times\Omega_{\mathbf v})}^2
+\delta \norm{\sigma_{tr}^{1/2}\nabla_{\mathbf v} u_0}
_{L_2(I_\perp\times\Omega_{\mathbf v})}^2\Big).
\end{equation}
\end{proposition}
\begin{remark}
Due to the size of smallness parameters $\delta$ and $\sigma_{tr}$ 
we can easily verify the second stability estimate \eqref{SSD1M} 
which also yields
\begin{equation}\label{SSD1N}
\norm{u_x+ {\mathbf v}\cdot \nabla_\perp u}_{L_2(\Omega)}\le \tilde C 
\norm{\sigma_{tr}^{1/2}\nabla_{\mathbf v} u_0}
_{L_2(I_\perp\times\Omega_{\mathbf v})}.
\end{equation}
Hence, using the equation \eqref{Fermi7}, we get  
\begin{equation}\label{SSD1P}
\norm{\sigma_{tr}\Delta_{\mathbf v} u}_{L_2(\Omega)}\le \tilde C 
\norm{\sigma_{tr}^{1/2}\nabla_{\mathbf v} u_0}
_{L_2(I_\perp\times\Omega_{\mathbf v})}. 
\end{equation}
The estimate \eqref{SSD1P} indicates the regularizing effect of the 
diffusive term $\Delta_{\mathbf v} u$ in the sense that 
$u_0\in H_{\sqrt{\sigma_{tr}}}^r(I_\perp\times\Omega_{\mathbf v})$ implies that 
$u\in H_{\sigma_{tr}}^{r+1}(I_\perp\times\Omega_{\mathbf v})$. However this 
regularizing effect will decrease by the size of $\sigma_{tr}$. 
\end{remark}

\section{Model problems in lower dimensions}

\label{sec:fyra}

We consider now a forward peaked narrow radiation beam entering into the symmetric  
domain $I_y \times I_\eta = [-y_0 , y_0] \times
[-\eta_0, \eta_0]$;  $(y_0 , \eta_0) \in \mathbb{R}_+^2$ at $(0,0)$ and penetrating in
the direction of the positive $x$-axis.  Then the computational domain
$ \Omega $ of our study is a three dimensional slab with $(x, y, \eta)
\in \Omega = I_x \times I_y \times I_\eta$ where  $ I_x = [0, L]$.
In this way, the problem \eqref{Fermi7} will be transformed into the
following lower dimensional model problem
\begin{equation}
\label{model}
\left\{
\begin{array}{rl}
u_x + \eta u_y = \frac 12 \sigma_{tr} u_{\eta \eta}  \quad  &(x,y,\eta) \in \Omega, \\
u_\eta (x, y, \pm \eta_0) = 0 \quad &(x,y) \in I_x \times I_y, \\
u(0, y, \eta)  = f(y, \eta)  \quad &(y, \eta) \in I_y \times I_\eta,  \\
u(x, y, \eta) = 0 \quad & \mbox{ on }  \Gamma_\beta^- \setminus\{(0, y, \eta)\}.
\end{array}
\right.
\end{equation}
For this problem we implement two different versions of the 
streamline 
diffusion method: the semi-streamline 
diffusion and the characteristic streamline diffusion. 
Both cases are discretized using 
linear polynomial approximations.

\subsection{The semi-streamline diffusion method}
\label{sec:semiSD}

In this version we derive a discrete scheme for computing the 
approximate solution $ u_h $ of the exact solution $ u $ using the
SD-method for discretizing the $ (y, \eta) $-variables 
(corresponding to multiply the equation by test functions 
of the form $w+\delta\eta w_y$) 
combined with the backward Euler method for the $ x $-variable. 
We start by introducing the bilinear forms 
$a(\cdot, \cdot)$ and $b(\cdot, \cdot)$ for
the problem (\ref{model}) as: 
\begin{equation}
\label{FD1}
\begin{split}
a(u, w) = & \, (\eta u_y , w)_\perp + \delta(\eta u_y ,\eta w_y)_\perp
+\frac 12 (\sigma_{tr} u_\eta , w_\eta)_\perp \\ 
& +\frac 12 \delta(\sigma_{tr}u_\eta,  w_y+\eta w_{y\eta})_\perp
 -\frac 12\delta \int_{I_y}\sigma_{tr} \eta u_\eta w_y\Big|_{\eta=-\eta_0}^{\eta=\eta_0} dy, \\
b(u, w) =& \, ( u , w)_\perp + \delta(u , \eta w_y)_\perp,
\end{split}
\end{equation}
where $(\cdot,  \cdot)_\perp:=(\cdot,  \cdot)_{I_y\times I_\eta}$. 
Then the continuous problem reads as: for each $ x \in (0, L] $, find $ u(x, \cdot) \in H^1_\beta $ such that 
$$
b(u_x, w)+a(u, w)=0, \qquad \forall w\in H^1_\beta,
$$
where
$$
H^1_\beta:=\{ w\in H^1 (I_y\times I_\eta); 
w=0 \mbox{ on }  \Gamma_\beta^- \}, 
$$
and
\begin{equation}\label{BCLOWDIM}
\Gamma_\beta^- :=\{(y,\eta)\in \Gamma:=\partial ( I_y\times I_\eta), \,\,\,
\mbox{with}\,\, {\mathbf n}\cdot\beta<0\},
\end{equation} 
with $\beta=(\eta, 0)$. 
Then the semi-streamline diffusion method for the continuous problem
(\ref{model}) reads as follows: 
for each $ x \in (0, L] $, find $ u_h (x, \cdot)  \in \mathcal{V}_{h,\beta} $
such that,
\begin{equation}
\label{FD3}
b(u_{h, x} ,w) + a(u_h, w) = 0, \quad \forall w \in \mathcal{V}_{h,\beta},
\end{equation}
where $\mathcal{V}_{h,\beta} \subset H^1_\beta $ consists of continuous piecewise linear 
functions. 
Next, we write the global discrete solution 
by separation of variables as 
\begin{equation}
\label{FD4}
u_h(x,y,\eta) = \sum_{j=1}^{N} U_j(x) \phi_j(y,\eta), 
\end{equation}
where $ N $ is the number of nodes in the mesh. Letting $ w = \phi_i $ for $ i = 1, 2, \ldots, N $
and inserting (\ref{FD4}) into (\ref{FD3}) we get the following
discrete system of equations,
\begin{equation}
\label{FD5}
\sum_{j=1}^{N} U'_j(x) b(\phi_j, \phi_i) + \sum_{j=1}^{N} U_j(x) a(\phi_j, \phi_i) = 0,  \quad i = 1, 2, \ldots, N.
\end{equation}
Equation (\ref{FD5}) in matrix form can be written as
\begin{equation}
\label{FD6}
B U'(x)  +A U(x) = 0,
\end{equation}
with $ U=[U_1, ..., U_N]^T, B = (b_{ij}), b_{ij} =b(\phi_j , \phi_i)$
and $A = (a_{ij}), a_{ij} =a(\phi_j , \phi_i), i, j = 1, 2, \ldots, N
$. We apply now the backward Euler method for further 
discretization of the equation 
(\ref{FD6}) in variable $ x $, and with the step size $ k_m $, 
to obtain an iterative form viz 
\begin{equation}
\label{FD7}
B(U^{m+1} - U^{m}) + k_m A U^{m+1} =0.
\end{equation}
The equation above can be rewritten as a system of equations for finding the
solution $ U^{m+1}$ (at ``time'' level $x=x_{m+1}$) 
on iteration $ m+1 $ from the known solution $ U^{m} $
from the previous iteration step $ m$:
\begin{equation}
[ B + k_m A] U^{m+1} = B U^m.
\label{backward Euler}
\end{equation}

\subsection{Characteristic Streamline Diffusion Method}

In this part we construct an oriented phase-space mesh to obtain the 
characteristic
streamline diffusion method. Before formulating this method, 
we need to construct a new subdivision of 
$ \Omega = I_x \times I_y \times I_\eta $. To this end and 
for $ m = 1, 2, \ldots, M $, we define a subdivision of
$\Omega_m = [ x_{m-1}, x_m ] \times I_y \times I_\eta:=I_m\times I_y \times I_\eta $ into elements
\begin{equation*}
\hat{\tau}_m = \{ (x, y + (x - x_m) \eta, \eta ) :
(y, \eta) \in \tau \in \mathcal{T}_h, \,\, x \in I_m \},
\end{equation*}
where $ \mathcal{T}_h $ is a previous triangulation of $ I_\perp $.
Then we introduce, {\sl slabwise}, the function spaces
\begin{equation*}
\hat{\mathcal{V}}_m = \{ \hat{w} \in C (\Omega_m ) :
\hat{w}(x, y, \eta ) = w(y + (x - x_m)\eta, \eta), w \in \mathcal{V}_{h,\beta} \}.
\end{equation*}
In other words $ \hat{\mathcal{V}}_m $ consists of continuous functions
$ \hat{w} (x, y, \eta) $ on $ \Omega_m $ such that $ \hat{w} $ is constant along
characteristics $ (\hat{y}, \hat{\eta}) = (y + x\eta, \eta) $ 
parallel to the sides
of the elements $ \hat{\tau}_m $, meaning that the derivative in the 
characteristic direction: 
$ \hat{w}_x + \eta \hat{w}_y = 0 $.
The streamline diffusion method can now be reduced to the following 
formulation (where only the $\sigma_{tr}$-term survives):
find $ \hat{u}_h $ such that, for each $ m = 1, 2, \ldots, M,$ 
 $ \hat{u}_h|_{\Omega_m} \in \hat{\mathcal{V}}_m $ and
\begin{equation}
\label{CSD}
\begin{split}
\frac 12 \int_{\Omega_m}\sigma_{tr} \hat{u}_{h, \eta} w_\eta \, dx dy d\eta
+ \int_{I_\perp} \hat{u}_{h,+} (x_{m-1}, y, \eta) w_+ (x_{m-1}, y, \eta) 
\, dy d\eta \\
= \int_{I_\perp} \hat{u}_{h,-} (x_{m-1}, y, \eta) w_+ (x_{m-1}, y, \eta) \, dy d\eta,
\quad \forall w \in \hat{\mathcal{V}}_m.
\end{split}
\end{equation}
Here, for definition of $ \hat{u}_{h,+},\hat{u}_{h,-}, w_+ $ we refer to (\ref{Fermi12}).

\section{Adaptive algorithm}

\label{sec:ref}

In this section we formulate an adaptive algorithm, which is used in
computations of the numerical examples studied in Section \ref{sec:num}.
This algorithm improves the accuracy of the computed
solution $ u_h $ of the model problem (\ref{model}). In the sequel for 
simplicity we denote $I_y\times I_\eta$ also by $\Omega_\perp$ 
(this however, should not be mixed with the notation in the theoretical Sections 1-3).

\vskip 0.3cm 
\noindent 
 \textbf{The Mesh Refinement Recommendation} \emph{We refine the
   mesh in neighborhoods of those points in }$I_y\times I_\eta $\emph{\ where
   the error} $ \varepsilon_n = | u - u_h^n | $\emph{\ attains its maximal values. More
   specifically, we refine the mesh in such subdomains of
 }$I_y\times I_\eta $\emph{\ where}
\begin{equation*}
\varepsilon_n  \geq \widetilde{\gamma} \max \limits_{\Omega_\perp} \varepsilon_n.
\end{equation*}
\emph{Here $ \widetilde{\gamma} \in (0,1)$ is a number which should
be chosen computationally and $ u_h^n $ denotes the computed solution
on the $ n $-th refinement of the mesh.}

\vskip 0.5cm 
\noindent
\textbf{The steps in adaptive algorithm}

\vspace{0.3cm}

\begin{itemize}

\item[Step 0.]  Choose an initial mesh $ I_m^0 \times
 \tau^0 $ in $ I_x \times I_y\times I_\eta $ and obtain the numerical solution
  $ u_h^n,\, n >0 $, where $n$ is  number of the mesh refinements, in the 
  following steps:

\item[Step 1.]  Compute the numerical solution $ u_h^n$ on
  $ \tau^n$ using any of the finite element methods introduced in section
  \ref{sec:fyra}.

\item[Step 2.]  Refine those elements in the mesh $ \tau^n $ for which
\begin{equation} \label{adaptalg}
\varepsilon_n  \geq \widetilde{\gamma} \max \limits_{\Omega_\perp} \varepsilon_n.
\end{equation}
Here, the values for the tolerance $ \widetilde{\gamma} \in \left(
0,1\right) $ are chosen by the user.

\item[Step 3.]  Define a new refined mesh as $ \tau^{n+1} $.
  Construct a new partition $ I_m^{n+1}$ if needed.  Perform
  steps 1-3 on the mesh $  I_m^{n+1} \times \tau^{n+1} $.
Stop mesh refinements when $\| u_h^n - u_h^{n-1} \|_{L_2(\Omega_\perp)} < tol$,
where $tol$ is a total tolerance chosen by the user.

\end{itemize}

\begin{table}
\center
\begin{tabular}{ l | l l l l l l}
\hline
Nr. of &  Nr. of   & Nr. of & DOF & $e_n = \| u - u^n_h\|_{L_2}$  & $e_n/e_{n+1}$\\
 refinement, $n$& elements & vertices &   &  & \\
\hline
0   & 272      & 157    & 157     & 8.364e-03    &     \\
1              & 1176     & 591    & 591     & 2.134e-04    & 3.92\\
2              & 4704     & 2268   & 2268    & 5.368e-04    & 3.97\\
3              & 17616    & 8878   & 8878    & 1.345e-05    & 3.99\\
4              & 69864    & 35231  & 35231   & 1.263e-05    & 4.00\\
\hline
\end{tabular}
\caption{Test 1-a).  Computed errors $e_n = \|u -
  u_h^n \|_{L_2(\Omega_\perp)}$ and $e_n/e_{n+1}$ on the adaptively
  refined meshes. Here, the solution $ u_h^n $ is computed using
  semi-streamline diffusion method of section \ref{sec:semiSD} with
  $\widetilde{\gamma} = 0.5 $ in the adaptive algorithm and $\alpha=0.1$ in (\ref{dirac}).}
\label{test1ap1}
\end{table}

\begin{table}
\center
\begin{tabular}{ l | l l l l l l}
\hline
Nr. of &  Nr. of   & Nr. of & DOF & $e_n = \| u - u_h^n \|_{L_2}$  & $e_n/e_{n+1}$\\
 refinement, $n$& elements & vertices &   &  & \\
\hline
0    & 272      & 157    & 157     & 8.364e-03     &    \\
1              & 1088     & 585    & 585     & 8.278e-03    & 1.01\\
2              & 4352     & 2257   & 2257    & 2.105e-03    & 3.93\\
3              & 17408    & 8865   & 8865    & 5.290e-04    & 3.98\\
4              & 69632    & 35137  & 35137   & 1.325e-04    & 3.99\\
\hline
\end{tabular}
\caption{Test 1-b). Computed  errors $ e_n = \| u -
  u_h^n \|_{L_2(\Omega_\perp)}$ and $e_n/e_{n+1}$ on the adaptively
  refined meshes. Here, the solution $ u_h^n $ is computed using
   semi-streamline diffusion method of section \ref{sec:semiSD} with $ \widetilde{\gamma} = 0.7 $
 in the adaptive algorithm and $\alpha=0.1$ in (\ref{dirac}). }
\label{test1bp1}
\end{table}

\section{Numerical examples}

\label{sec:num}

In this section we present numerical examples which show the
performance of  an adaptive finite element method
 for
the solution of the model problem (\ref{model}).  Here, all 
computations are performed in Matlab COMSOL Multiphysics using module
LIVE LINK MATLAB.
  We choose the domain $ \Omega_\perp = I_y \times I_\eta$ as 
 \begin{equation*}
 \Omega_\perp = \left\{ (y, \eta) \in (
 -1.0,1.0) \times (-1.0,1.0)\right\} .
 \end{equation*}
Our tests are performed with a fixed diffusion coefficient
$\sigma_{tr} = 0.002$. Further, due to smallness of the parameters $\delta$ and $\sigma_{tr}$, 
the terms that involve the product $\delta\sigma_{tr}$ are assumed to be negligible.
In the backward Euler scheme, used discretization in $x$-variable, we 
solve the system of equations (\ref{backward Euler}) which ends up with a 
discrete (computed) solution $U^{m+1}$ of (\ref{model}) at the time
iteration $m+1$ and with the time step $k_m$ 
which has been chosen to be $ k_m = 0.01$.

Previous computational studies, e.g.\ \cite{MA-EL:2007}, have shown oscillatory
behavior of the solution $ u_h $ when the semi-streamline diffusion
method was used, and layer behavior  when the standard Galerkin method
was applied to solve the model problem (\ref{model}).
In this work we significantly improve results of \cite{MA-EL:2007} by
using the adaptive algorithm of section \ref{sec:ref} on the locally
adaptively refined meshes. 
All our computations are compared with the closed form analytic 
solution for the model problem (\ref{model})  given by
 \begin{equation}
  \label{Exact solution}
  u (x,y, \eta) = \frac{ \sqrt{3}}{ \pi \sigma_{tr} x^{2}}
\exp \left[ - \frac{2}{\sigma_{tr}} \left( \frac{3 y^2}{x^3} - \frac{3 y \eta}{x^2}+ \frac{\eta^2}{x} \right)  \right],
  \end{equation}
when the initial data is given by $ u (0, y, \eta) = \delta (y) \delta (\eta) $.
\\

We have performed the following   computational tests:

\begin{itemize}

\item Test 1. Solution of the model problem (\ref{model}) with a ``Dirac type''
  initial condition 
\begin{equation}\label{dirac}
u(0, y,\eta) = f(y,\eta)  = 1/( y^2 + \eta^2 +\alpha), 
\quad (y, \eta) \in \Omega_\perp, 
\end{equation}
for   different values of the parameter $\alpha \in (0,1)$.

\item Test 2. Solution of the model problem (\ref{model}) with
  ``Maxwellian type'' initial condition 
 \begin{equation}\label{maxwellian} 
u(0, y, \eta) = f(y, \eta)   = \exp^{-(y^2 + \eta^2 + \alpha)}, 
\quad (y, \eta) \in \Omega_\perp, 
\end{equation}
for different values of $\alpha \in (0,1)$.

\item Test 3. Solution of the model problem (\ref{model}) with
 a ``hyperbolic type'' initial condition 
 \begin{equation}\label{hyperbolic} 
u (0,y, \eta ) = f(y, \eta)   = \frac{1}{\sqrt{y^2+ \eta^2+ \alpha}}, 
\quad (y, \eta) \in \Omega_\perp, 
\end{equation}
for $\alpha= 0.19$. 

\end{itemize}

 \begin{figure}[tbp]
 \begin{center}
 \begin{tabular}{ccc}
 {\includegraphics[scale=0.25, clip=true,]{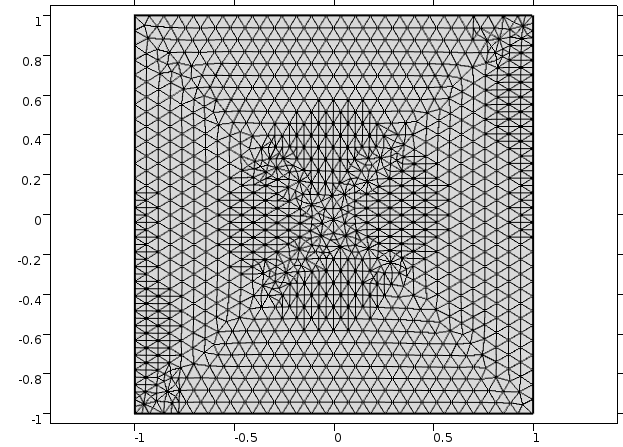}} &
 {\includegraphics[scale=0.25, clip=true,]{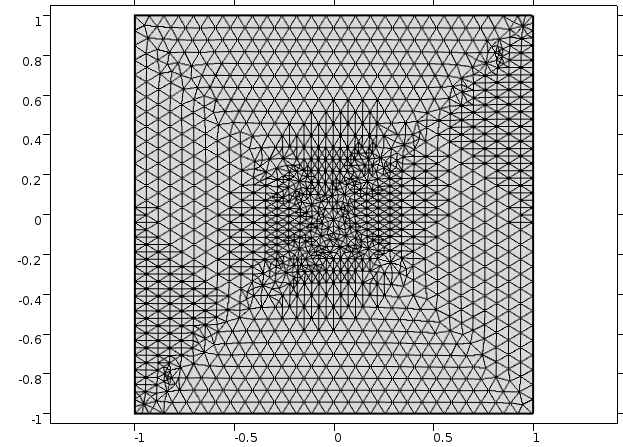}} &
 {\includegraphics[scale=0.25,clip=true,]{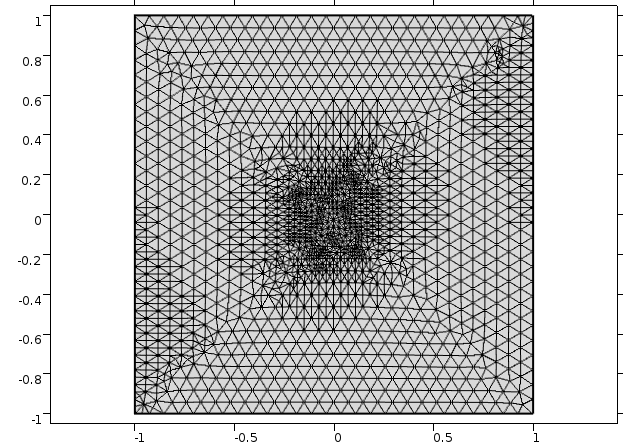}} \\
a)  &  b)  & c)  \\ 
 {\includegraphics[scale=0.25, clip=true,]{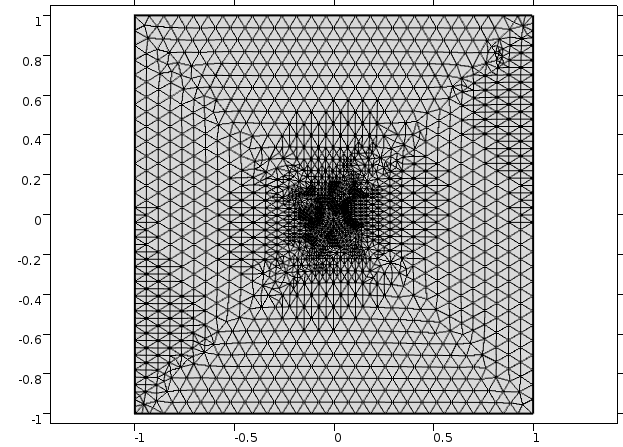}} &
{\includegraphics[scale=0.25, clip=true,]{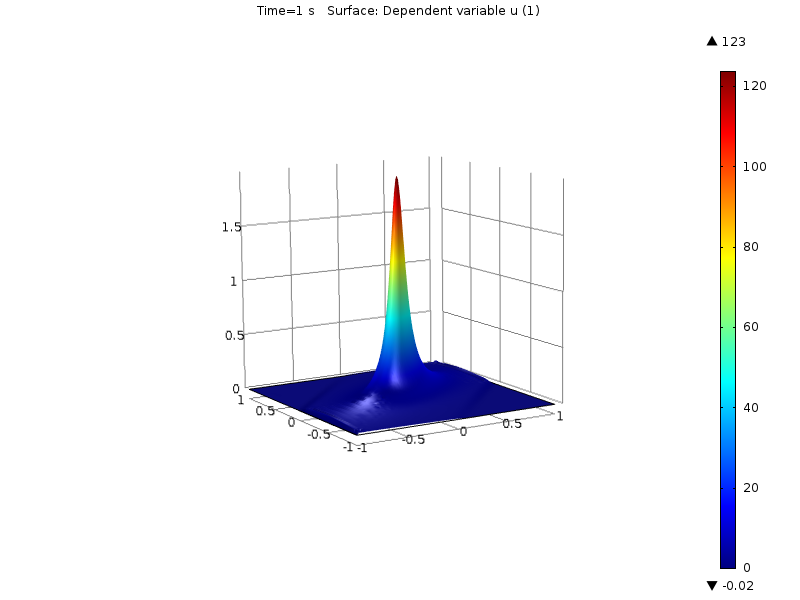}} \\
d)  &  e)
 \end{tabular}
 \end{center}
 \caption{Test 1-a. a)-d) Locally adaptively refined meshes of Table
   \ref{test1ap1}; e) Computed solution on the 4 times adaptively refined mesh d). }
 \label{fig:fig1a}
 \end{figure}

 \begin{figure}[tbp]
 \begin{center}
 \begin{tabular}{ccc}
 {\includegraphics[scale=0.3, clip=true,]{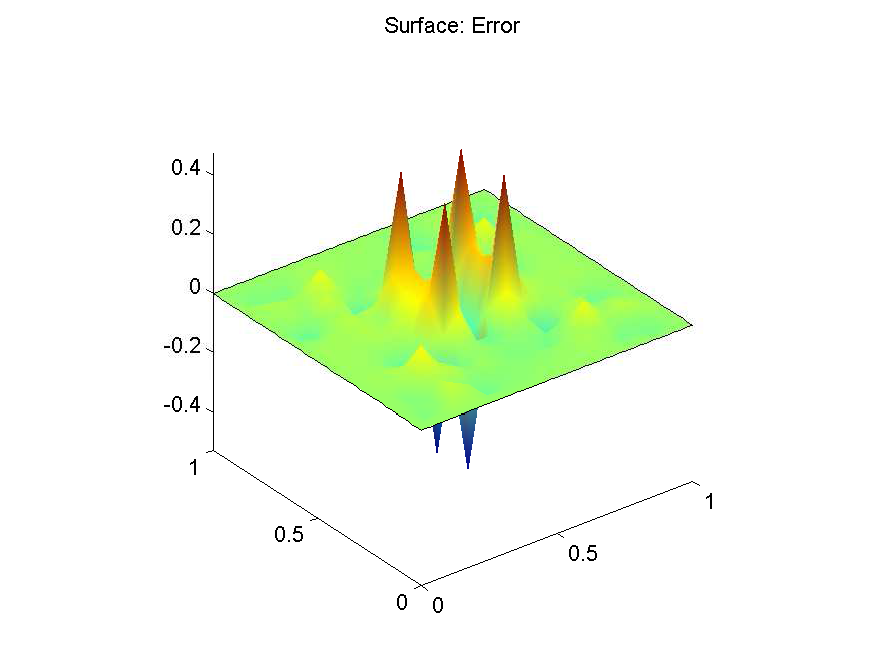}} &
 {\includegraphics[scale=0.3, clip=true,]{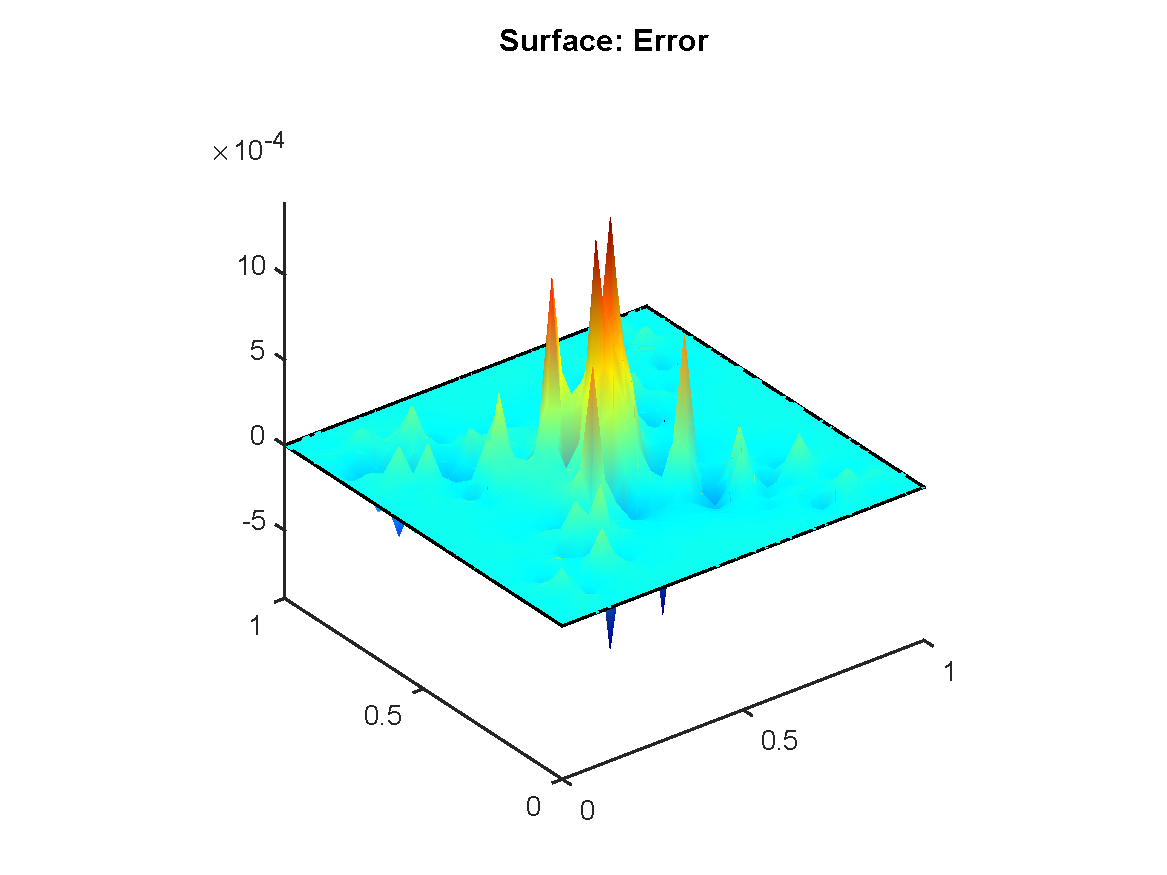}} &
 {\includegraphics[scale=0.3,clip=true,]{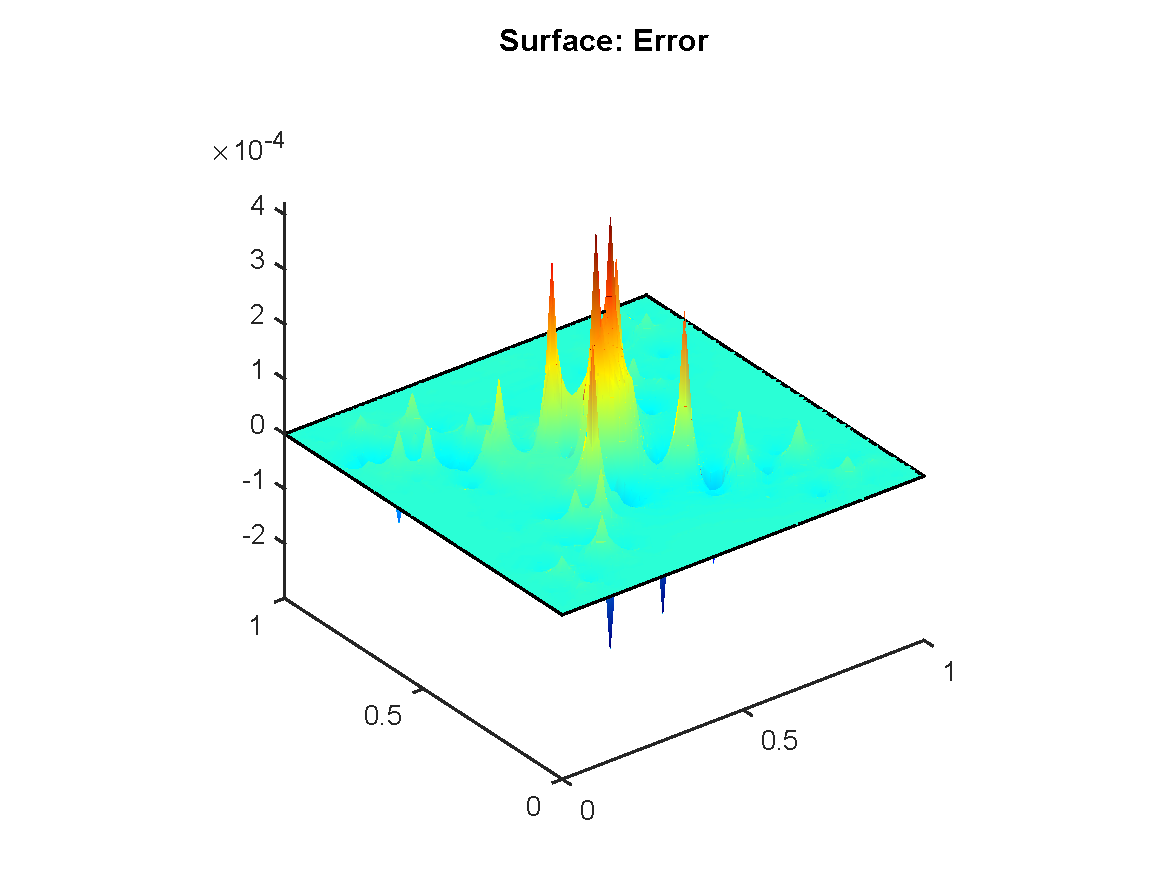}} \\
a)   &  b)  & c) \\
 {\includegraphics[scale=0.3, clip=true,]{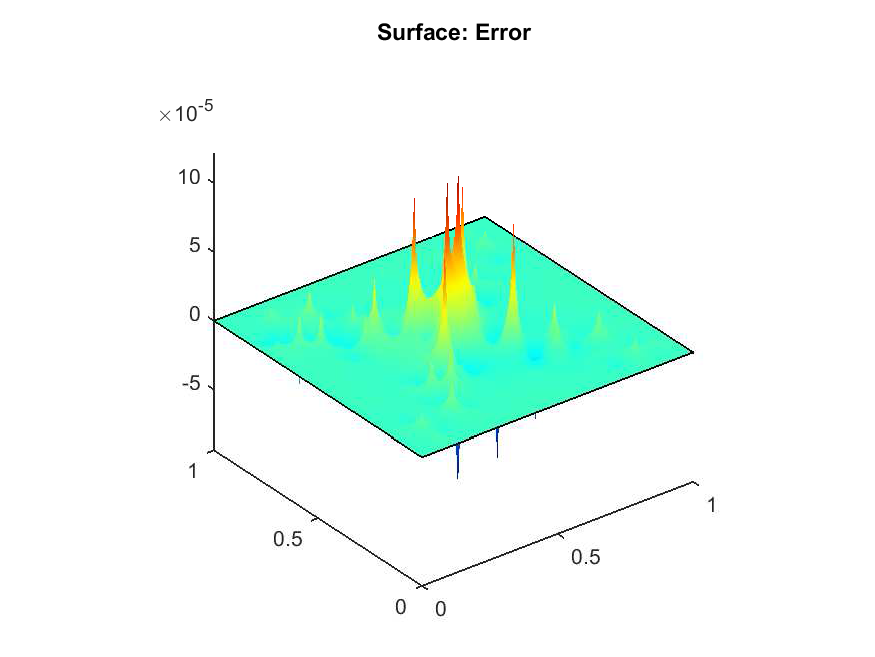}} &
 {\includegraphics[scale=0.3, clip=true,]{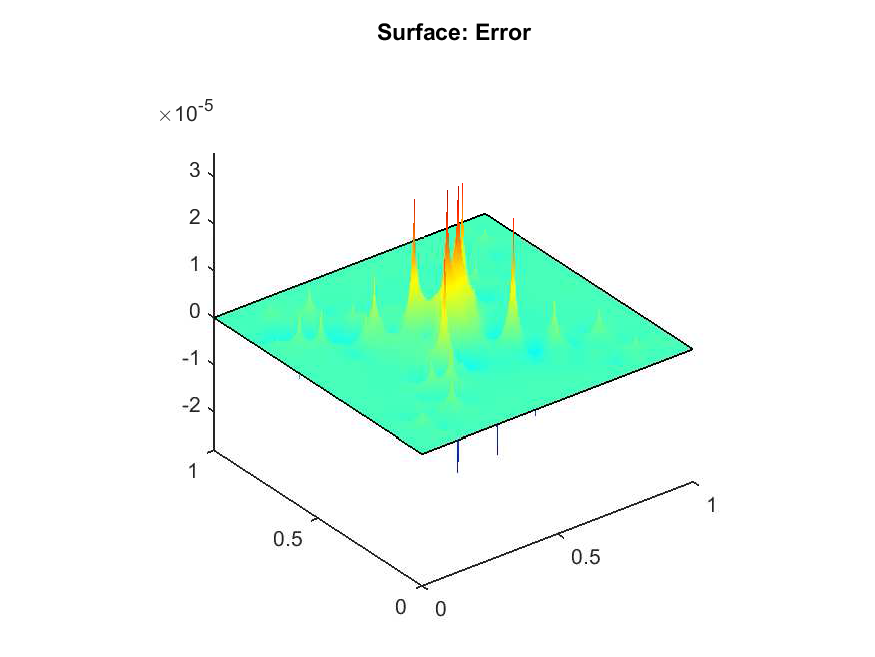}} &  \\
 \\
d)  &  e)  &  \\ 
 \end{tabular}
 \end{center}
 \caption{Test 1-a). Computed errors $ \mathcal{E}_n (x,y) =  u (x,y) - u_h^n (x,y) $ on the locally adaptively
   refined meshes of Table \ref{test1ap1}   on the meshes
   of Figure \ref{fig:fig1a}-a)-d).  }
 \label{fig:fig1ameshes}
 \end{figure}

 \begin{figure}[tbp]
 \begin{center}
 \begin{tabular}{ccc}
 {\includegraphics[scale=0.25, clip=true,]{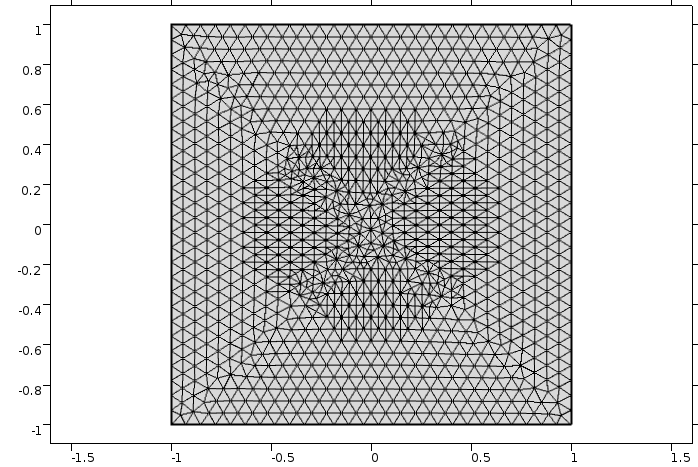}} &
 {\includegraphics[scale=0.25, clip=true,]{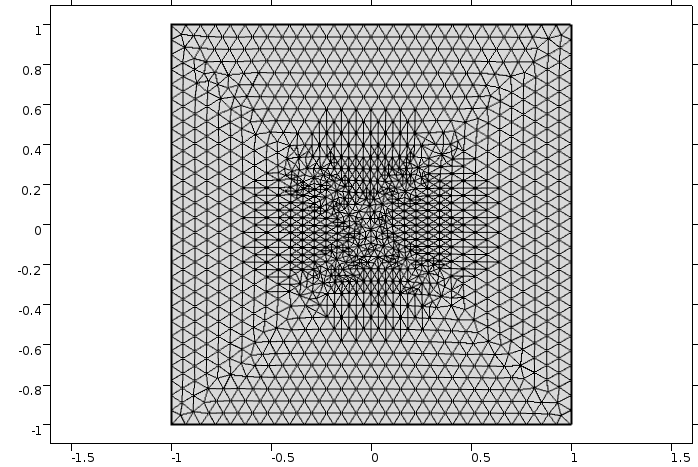}} &
 {\includegraphics[scale=0.25,clip=true,]{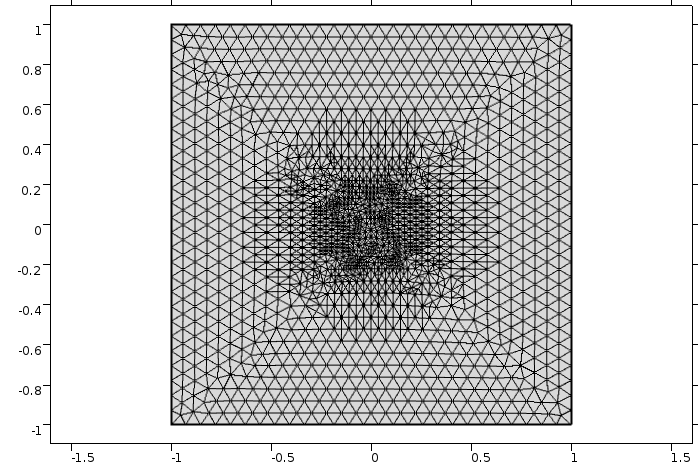}} \\
a)  &  b)  & c)  \\ 
 {\includegraphics[scale=0.25, clip=true,]{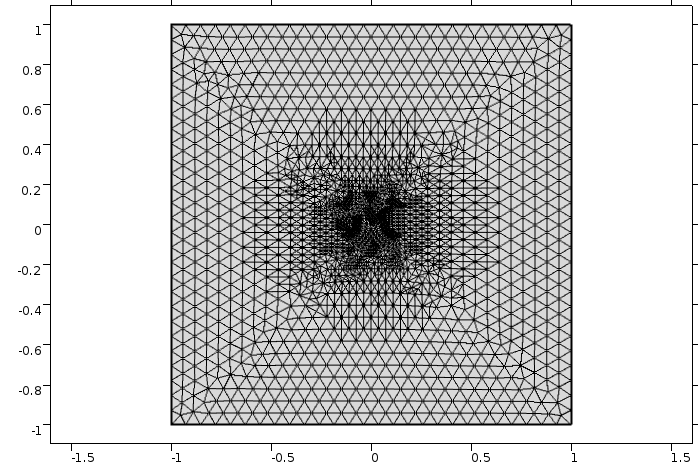}} &
{\includegraphics[scale=0.25, clip=true,]{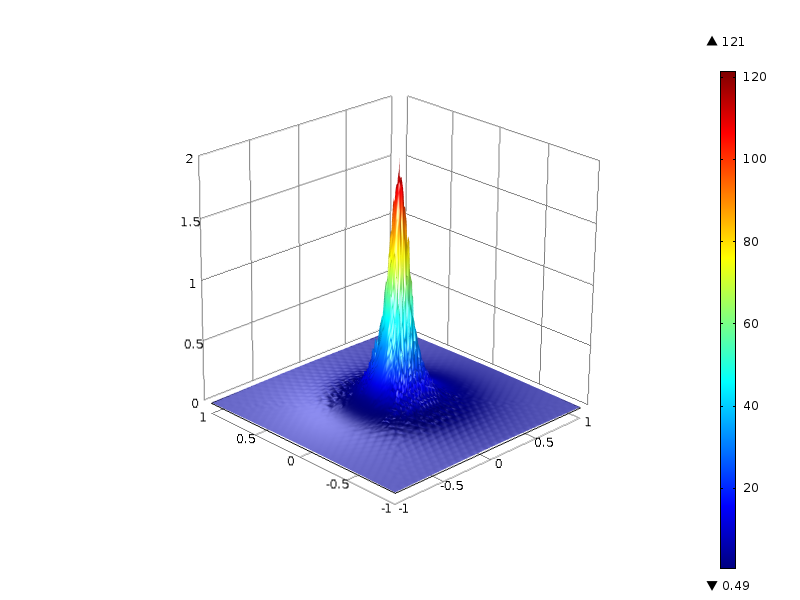}} \\
d)  &  e)
 \end{tabular}
 \end{center}
 \caption{Test 1-b.  a)-d) Locally adaptively refined meshes of Table
   \ref{test1bp1}; e) Computed solution on the 4 times adaptively refined mesh d). }
 \label{fig:fig1b}
 \end{figure}

\subsection{Test 1}
 
In this test we compute numerical simulations for the problem
(\ref{model}) with a ``Dirac type'' initial condition (\ref{dirac}) and for
different values of the parameter $\alpha \in (0,1)$ 
in (\ref{dirac}), where we use adaptive
algorithm of Section \ref{sec:fyra} on the locally adaptively refined
meshes. These meshes were refined according to the
error indicator (\ref{adaptalg}) in the adaptive algorithm. 
For computation of the finite element solution we employ 
semi-streamline diffusion method of Section \ref{sec:semiSD}.
We performed two set of numerical experiments:

\begin{itemize}
\item Test 1-a).  We take $\widetilde{\gamma} = 0.5$ in
  (\ref{adaptalg}). This choice of the parameter allows
to  refine the mesh $ \tau $ not only at the
  center of the domain $ \Omega_\perp$, but also at the boundaries of  
$ \Omega_\perp$.

\item Test 1-b). We take $\widetilde{\gamma} = 0.7$ in
  (\ref{adaptalg}). Such choice of the parameter allows to 
refine the mesh $ \tau $ only at the
  middle of the domain $ \Omega_\perp $.

\end{itemize}

Our computational tests have shown that the values for $\alpha \in
(0.05,0.1)$ give smaller computational errors 
$e_n = \| u -u_h^n \|_{L_2(\Omega_\perp)}$ than the other $\alpha$-values.  
The results of the computations for
$\alpha=0.1$ are presented in Tables \ref{test1ap1} and
\ref{test1bp1} for Test 1-a) and Test 1-b), respectively.  Using
these tables and Figures \ref{fig:fig1a} and \ref{fig:fig1b} we observe
that we have obtained significant reduction of the computational error
$e_n = \| u - u_h^n \|_{L_2(\Omega_\perp)}$ on the adaptively refined 
meshes. These errors 
in the form $ \mathcal{E}_n (x,y) = u (x,y) - u_h^n (x,y) $, 
on the adaptively refined meshes, are shown on Figure
\ref{fig:fig1ameshes}.  Using Tables \ref{test1ap1}, \ref{test1bp1} we
observe that the reduction of the computational error is faster and more
significant in the case a) than in the case b).  Thus, choosing the
parameter $\widetilde{\gamma} = 0.5$ in (\ref{adaptalg}) gives a 
better computational result and smaller error $e_n$ 
than $\widetilde{\gamma} = 0.7$. 
This allows us to conclude that the
refinement of the mesh $ \tau $ not only at the center of the domain 
$ \Omega_\perp $, but also at the boundaries of $ \Omega_\perp$ 
give significantly
smaller computational error $ e_n = \| u - u_h^n \|_{L_2(\Omega_\perp)} $.

We present the final solution $ u_h^4 $
computed on the 4 times adaptively refined mesh on the Figure
\ref{fig:fig1a}-f) for Test 1-a) and on the Figure \ref{fig:fig1b}-f)
for Test 1-b).  We note that in both cases we have obtained smoother  
computed solution $ u_h^4 $ without any oscillatory behavior.  This is
a significant improvement of the result of \cite{MA-EL:2007} where mainly 
oscillatory solution could be obtained.

\subsection{Test 2}

\begin{table}
\center
\begin{tabular}{ l | l l l l l l}
\hline
Nr. of &  Nr. of   & Nr. of & DOF & $ e_n = \| u - u_h^n \|_{L_2}$  & $e_n/e_{n+1}$\\
 refinement, $n$& elements & vertices &   &  & \\
\hline
0   & 272    & 157    & 585     & 2.582e-04      &         \\
1           & 1288   & 592    & 2267    & 3.242e-05          & 7.97\\
2           & 4552   & 2267   & 8911    & 4.062e-06          & 7.98\\
3           & 17628  & 8911   & 35432   & 5.085e-07          & 7.99\\
4           & 69812  & 35432  & 141612  & 4.362e-07          & 7.99\\
\hline
\end{tabular}
\caption{Test 2-a).  Computed values of errors $e_n = \| u -
  u_h^n \|_{L_2(\Omega_\perp)}$ and $e_n/e_{n+1}$ on the adaptively
  refined meshes. Here, the solution $ u_h^n $ is computed using
  characteristic streamline diffusion method with $ \widetilde{\gamma} = 0.5 $ in the adaptive algorithm. }
\label{test2ap1}
\end{table}

\begin{table}
\center
\begin{tabular}{ l | l l l l l l}
\hline
Nr. of &  Nr. of   & Nr. of & DOF & $e_n = \| u - u_h^n \|_{L_2}$  & $e_n/e_{n+1}$\\
 refinement, $n$& elements & vertices &   &  & \\
\hline
0   & 272    & 157    & 585     & 2.582e-03          &     \\
1           & 1088   & 585    & 2257    & 2.706e-03          & 7.81\\
2           & 4352   & 2257   & 8865    & 3.427e-04          & 7.90\\
3           & 17408  & 8865   & 35137   & 4.307e-05          & 7.96\\
4           & 69632  & 35137  & 139905  & 5.398e-06          & 7.98\\

\hline
\end{tabular}
\caption{Test 2-b).  Computed values of errors $e_n = \| u -
  u_h^n \|_{L_2(\Omega_\perp)}$ and $e_n/e_{n+1}$ on the adaptively
  refined meshes. Here, the solution $ u_h^n $ is computed using
 characteristic streamline diffusion method with $ \widetilde{\gamma} = 0.7 $ in the adaptive algorithm.}
\label{test2bp1}
\end{table}


 \begin{figure}[tbp]
 \begin{center}
 \begin{tabular}{ccc}
 {\includegraphics[scale=0.25, clip=true,]{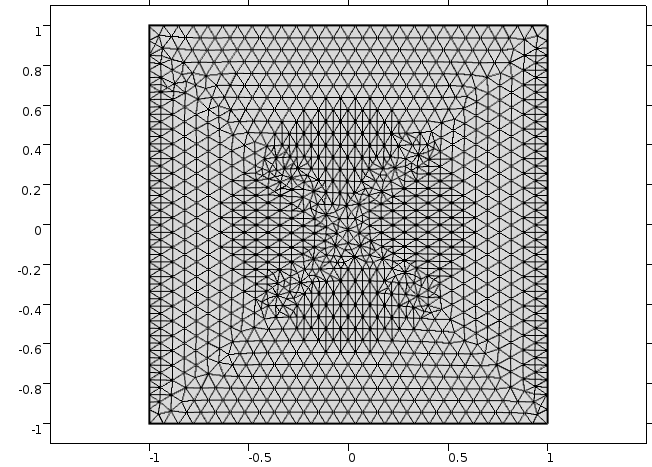}} &
 {\includegraphics[scale=0.25, clip=true,]{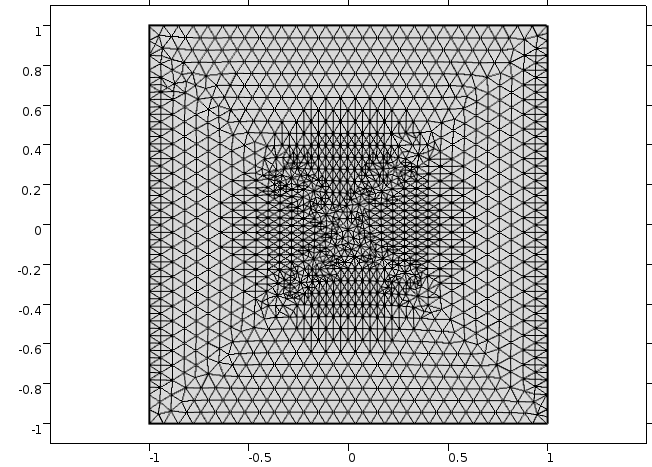}} &
 {\includegraphics[scale=0.25,clip=true,]{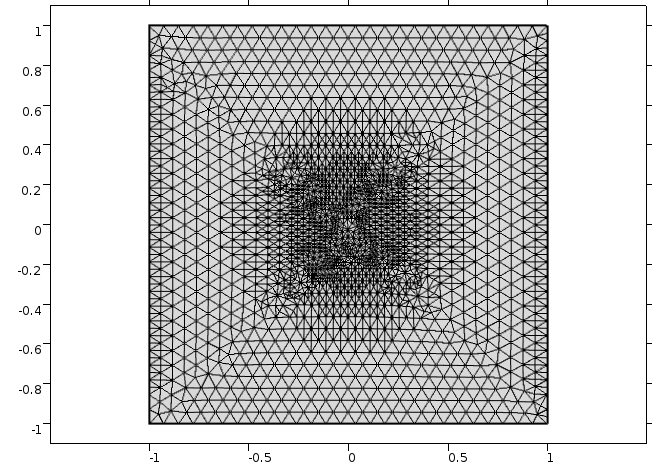}} \\
a)  &  b)  & c)  \\ 
 {\includegraphics[scale=0.25, clip=true,]{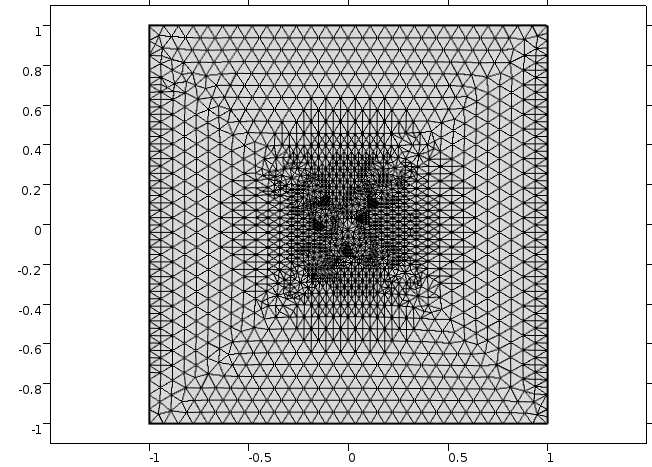}} &
{\includegraphics[scale=0.25, clip=true,]{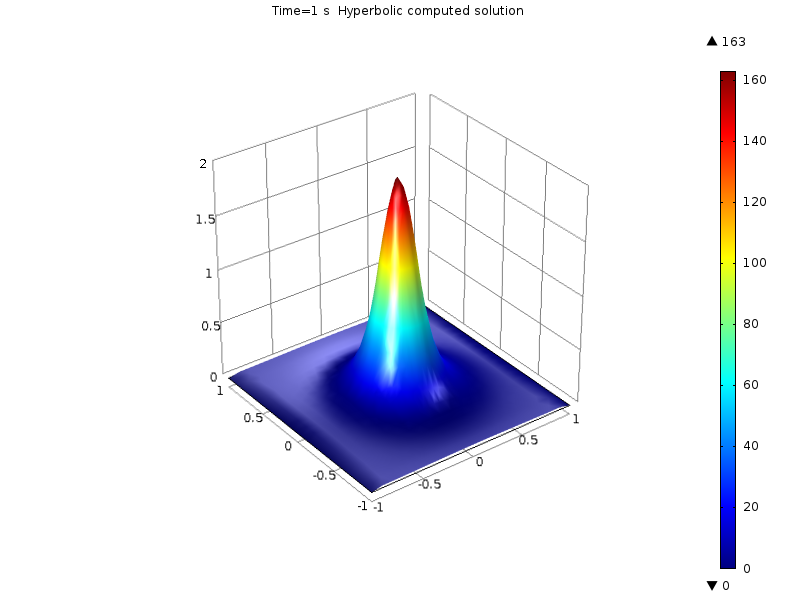}} \\
d)  &  e)
 \end{tabular}
 \end{center} 
 \caption{Test 2-a).  a)-d) Locally adaptively refined meshes of Table
   \ref{test2ap1}; e) Computed solution on the 4 times adaptively refined mesh d). }
 \label{fig:fig2a}
 \end{figure}

 \begin{figure}[tbp]
 \begin{center}
 \begin{tabular}{ccc}
 {\includegraphics[scale=0.25, clip=true,]{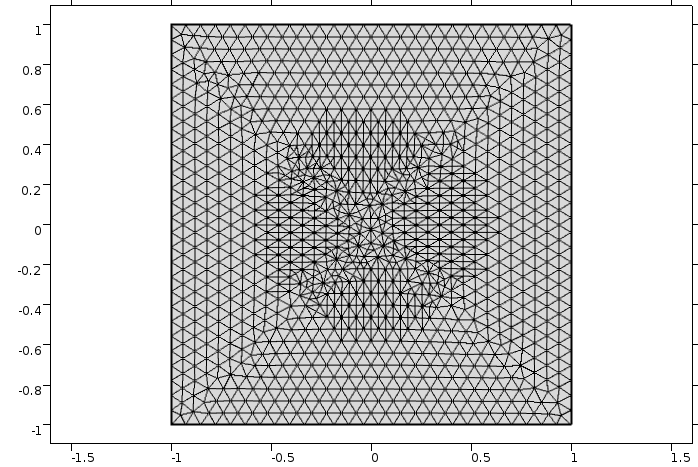}} &
 {\includegraphics[scale=0.25, clip=true,]{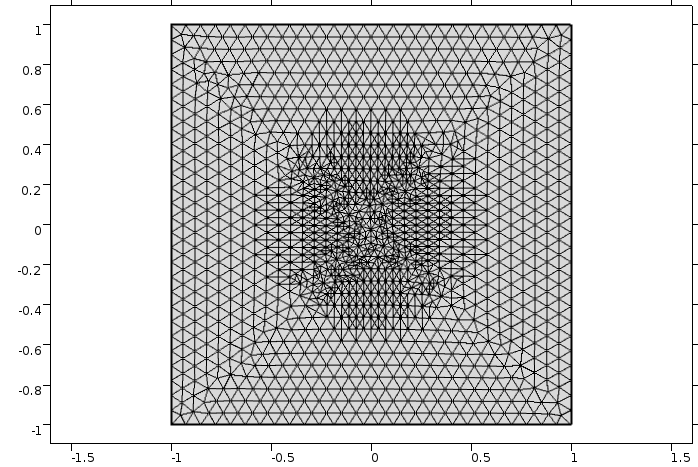}} &
 {\includegraphics[scale=0.25,clip=true,]{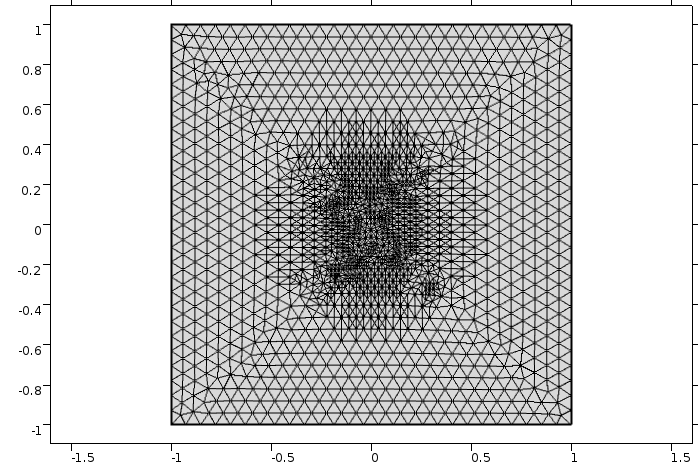}} \\
a)  &  b)  & c)  \\ 
 {\includegraphics[scale=0.25, clip=true,]{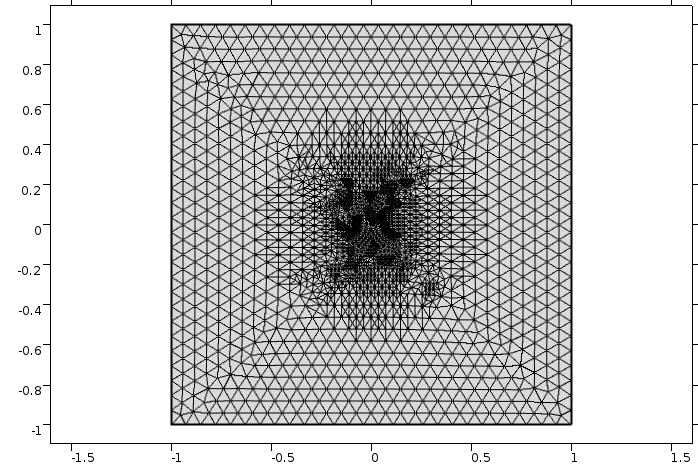}} &
{\includegraphics[scale=0.25, clip=true,]{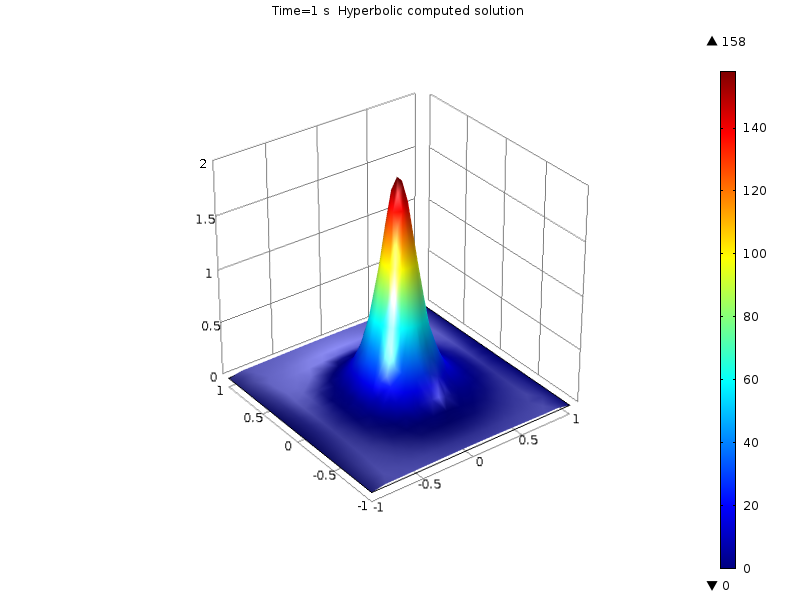}} \\
d)  &  e)
 \end{tabular}
 \end{center}
 \caption{Test 2-b).  a)-d) Locally adaptively refined meshes of Table
   \ref{test2bp1}; e) Computed solution on the 4 times adaptively refined mesh d). }
 \label{fig:fig2b}
 \end{figure}

In this test we perform numerical simulations for the problem
(\ref{model}) with Maxwellian initial condition (\ref{maxwellian})
and for different values of the parameter $\alpha \in (0,1)$. 
Again we use the
error indicator (\ref{adaptalg}) in the adaptive algorithm for local
refinement of meshes and perform two set of tests as in the case of 
Test 1 and with the same values on the parameter $\widetilde{\gamma}$.

For finite element discretization we use semi-streamline diffusion
method as in the Test 1.  To be able to control the formation of the layer 
which appears at the central 
point $ (y, \eta) = (0, 0) $ we use different values
of $\alpha \in (0,1)$ inside the function (\ref{maxwellian}).
 Our computational tests show that the value of the parameter 
$\alpha= 0.19$ is optimal one.

We present results of our computations for $\alpha= 0.19$ in Tables
\ref{test2ap1} and \ref{test2bp1} for Test 2-a) and Test 2-b),
respectively.  Using these Tables and Figures \ref{fig:fig2a} and 
\ref{fig:fig2b}, once again, we  observe significant reduction of the
computational error $e_n = \| u - u_h^n \|_{L_2(\Omega_\perp)}$ on the
adaptively refined meshes. 
 Using Tables \ref{test2ap1} and 
\ref{test2bp1}  again we observe  
 more significant reduction of the computational error 
 in the case a) then in the case b).  Thus,
choosing the parameter $\widetilde{\gamma} = 0.5$ in (\ref{adaptalg})
yields better computational results.

Final solution $ u_h^4$ computed on the 4
times adaptively refined mesh is shown on the Figure
\ref{fig:fig2a}-f) for Test 2-a) and on the Figure \ref{fig:fig2b}-f)
for Test 2-b). Again we observe that, with the above numerical 
 values for the parameters,  $\alpha$ and  $\widetilde{\gamma}$
 we have avoided the
formation of layers and in both tests we have obtained smooth computed
solution $ u_h^4$.  

\subsection{Test 3}

\begin{table}
\center
\begin{tabular}{ l | l l l l l l}
\hline
Nr. of &  Nr. of   & Nr. of & DOF & $e_n = \| u - u^n_h\|_{L_2}$  & $e_n/e_{n+1}$\\
 refinement, $n$& elements & vertices &   &  & \\
\hline
0  & 272   & 157   & 1285  & 1.565e-05    &     \\
1           & 1271  & 597   & 5115  & 9.732e-07   & 16.08\\
2           & 5084  & 2267  & 20937 & 6.052e-08   & 16.08\\
3           & 20336 & 9075  & 79825 & 3.771e-09   & 16.05\\
\hline
\end{tabular}
\caption{Test 3-a).  Computed values of errors $e_n = \| u -
  u_h^n \|_{L_2(\Omega_\perp)}$ and $e_n/e_{n+1}$ on the adaptively
  refined meshes. Here, the solution $ u_h^n $ is computed using
  semi-streamline diffusion method with $ \widetilde{\gamma} = 0.5 $ in the adaptive
  algorithm. }
\label{test3a}
\end{table}

\begin{table}
\center
\begin{tabular}{ l | l l l l l l}
\hline
Nr. of &  Nr. of   & Nr. of & DOF & $e_n = \| u - u^n_h\|_{L_2}$  & $e_n/e_{n+1}$\\
 refinement, $n$& elements & vertices &   &  & \\
\hline
0   & 272   & 157   & 1285  & 1.565e-05    &     \\
1           & 1088  & 585   & 5017  & 1.484e-06   & 15.98\\
2           & 4352  & 2257  & 19825 & 9.289e-07   & 16.02\\
3           & 17408 & 8865  & 78817 & 5.799e-08   & 16.02\\ 
\hline
\end{tabular}
\caption{Test 3-b).  Computed values of errors $e_n = \| u -
  u_h^n \|_{L_2(\Omega_\perp)}$ and $e_n/e_{n+1}$ on the adaptively
  refined meshes. Here, the solution $ u_h^n $ is computed using
  semi-streamline diffusion method  with $ \widetilde{\gamma} = 0.7 $ in the adaptive algorithm. }
\label{test3b}
\end{table}


\begin{figure}[tbp]
 \begin{center}
 \begin{tabular}{ccc}
 {\includegraphics[scale=0.25, clip=true,]{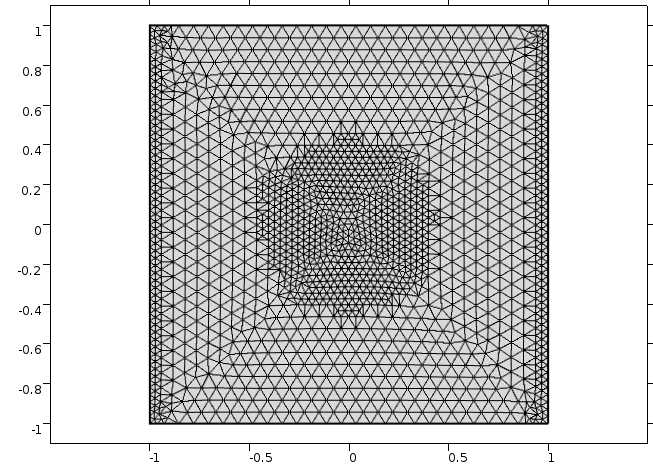}} &
 {\includegraphics[scale=0.25, clip=true,]{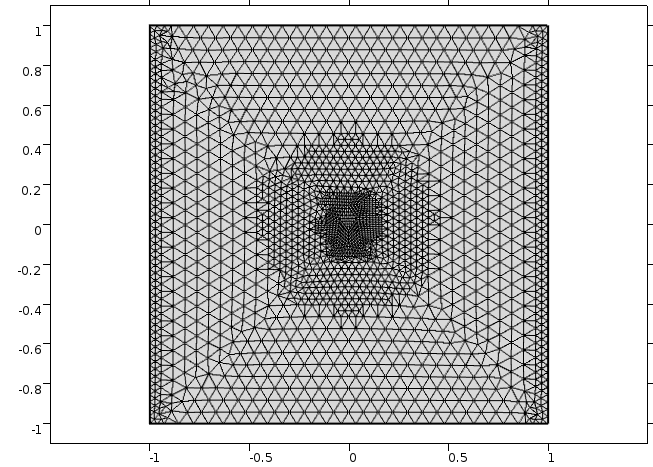}} &
 {\includegraphics[scale=0.25,clip=true,]{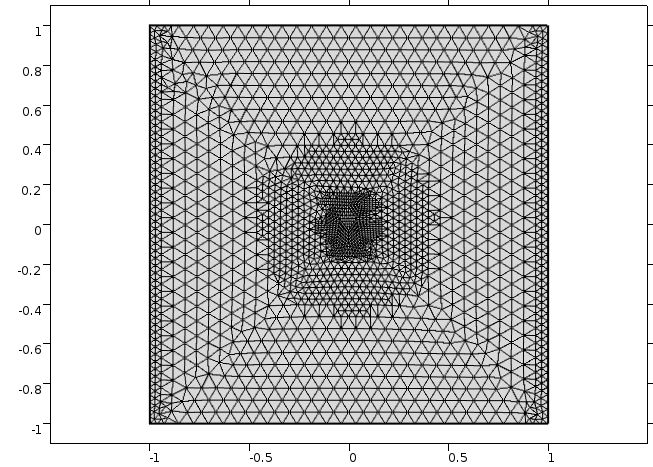}} \\
a)  &  b)  & c)  \\ 
 {\includegraphics[scale=0.25, clip=true,]{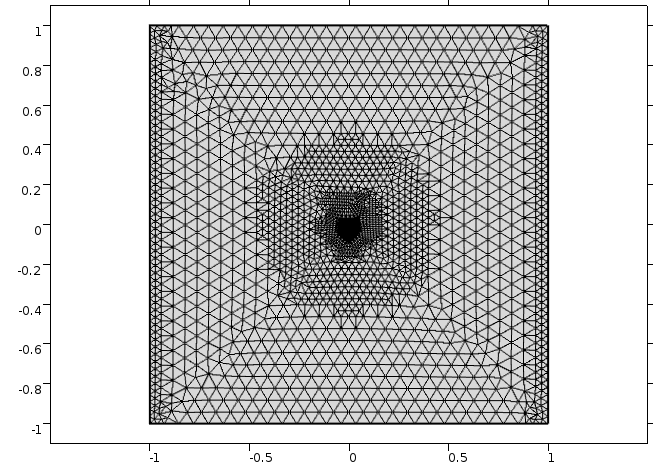}} &
{\includegraphics[scale=0.25, clip=true,]{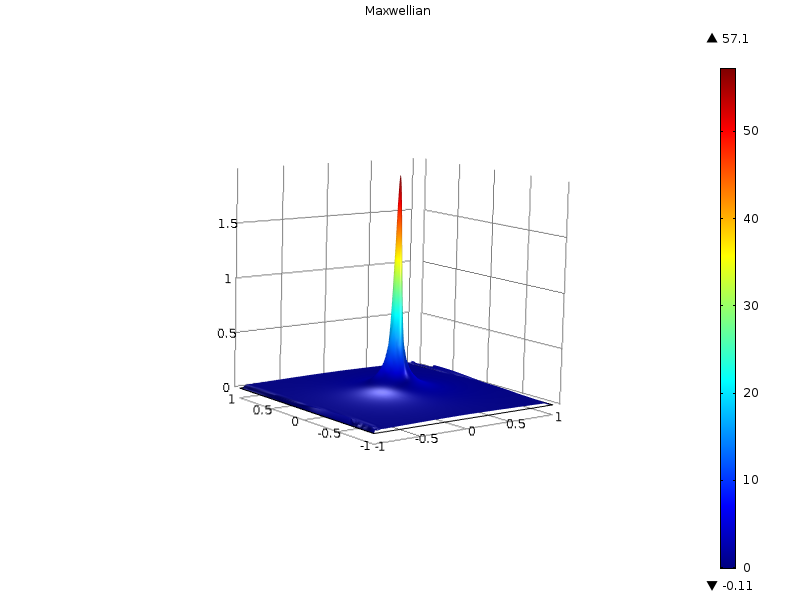}} \\
d)  &  e) 
 \end{tabular}
 \end{center}
 \caption{Test 3-a). a)-d) Locally adaptively refined meshes of Table
   \ref{test3a};  e) Computed solution on the 4 times adaptively refined mesh d).  }
 \label{fig:fig3a}
 \end{figure}

\begin{figure}[tbp]
 \begin{center}
 \begin{tabular}{ccc}
 {\includegraphics[scale=0.25, clip=true,]{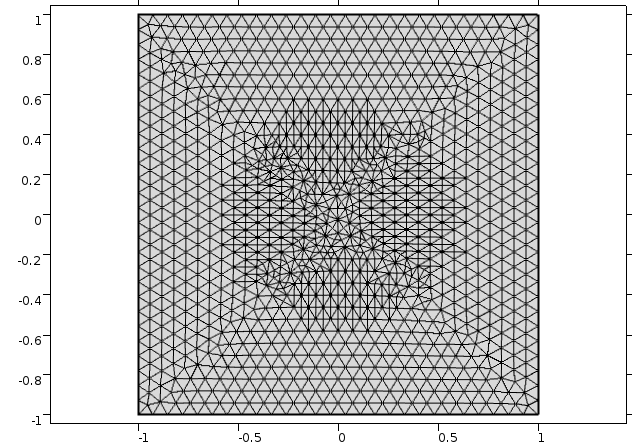}} &
 {\includegraphics[scale=0.25, clip=true,]{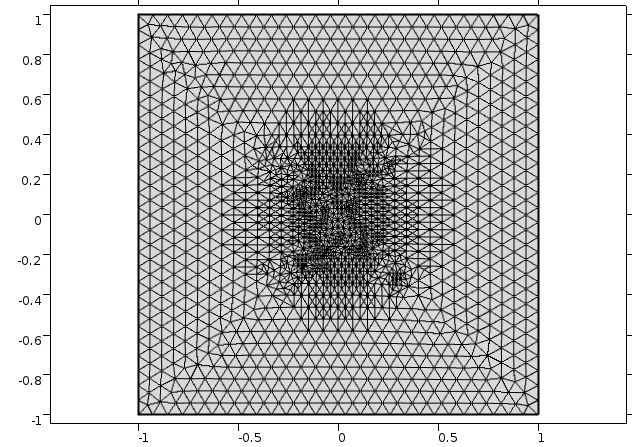}} &
 {\includegraphics[scale=0.25,clip=true,]{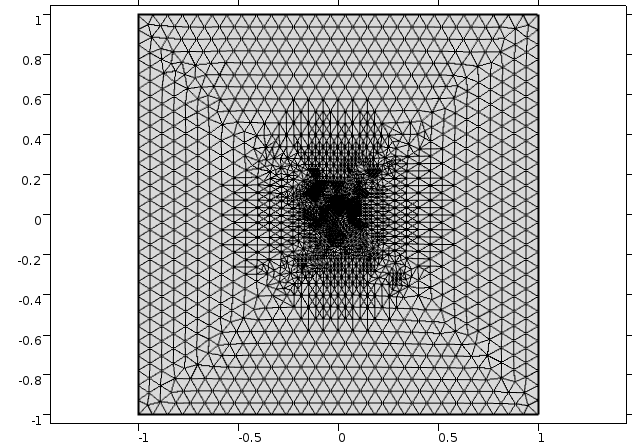}} \\
a)  &  b)  & c)  \\ 
 {\includegraphics[scale=0.25, clip=true,]{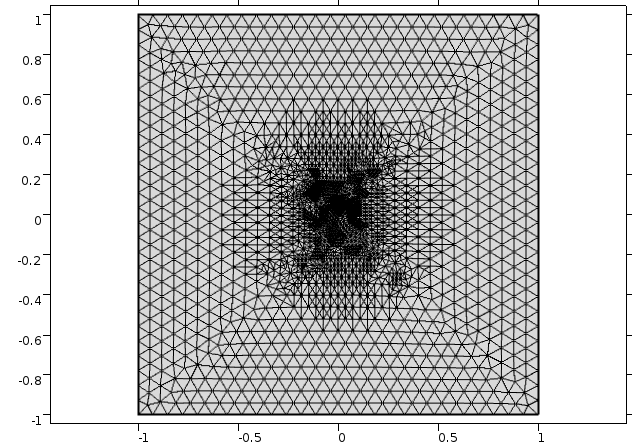}} &
{\includegraphics[scale=0.25, clip=true,]{computed_Max_2.png}} \\
d)  &  e) 
 \end{tabular}
 \end{center}
 \caption{Test 3-b). a)-d) Locally adaptively refined meshes of Table
   \ref{test3b};  e) Computed solution on the 4 times adaptively refined mesh d).  }
 \label{fig:fig3b}
 \end{figure}

In this test we perform numerical simulations of the problem
(\ref{model}) with hyperbolic initial condition (\ref{hyperbolic}) on
the locally adaptively refined meshes. Taking into account results of
our previous Tests 1,2 we take fixed value of $\alpha=0.19$ in
(\ref{hyperbolic}).  For finite element discretization we used the
semi-streamline diffusion method of Section \ref{sec:fyra}.  We again
perform two set of tests with different values of $\tilde{\gamma}$ in
(\ref{adaptalg}): in the Test 3-a) we choose $\tilde{\gamma}=0.5$, 
and in the Test 3-b) we  assign this parameter to be
$\tilde{\gamma}=0.7$.

We present results of our computations in Tables \ref{test3a} and 
\ref{test3b}.  Using these tables and Figures \ref{fig:fig3a} and 
\ref{fig:fig3b} we observe significant reduction of the computational
error $e_n = \| u - u_h^n \|_{L_2(\Omega_\perp)}$ on the adaptively
refined meshes.  Final solutions $ u_h^4 $ computed on the 4 times
adaptively refined meshes are shown on the Figure \ref{fig:fig3a}-f)
for the Test 3-a) and on the Figure \ref{fig:fig3b}-f) for the Test
3-b), respectively.

\section{Conclusion}

\label{Conclusion}

		Finite element method (FEM) is commonly used as
                numerical method for solution of PDEs. In this work
                FEM is applied to compute approximate 
solution of a, {\sl degenerate type}, 
                convection
                dominated convection-diffusion problem. We consider 
linear polynomial approximations and study 
                different finite element discretizations for the
                solutions for pencil-beam models based on Fermi and
                Fokker-Planck equations. First we have derived 
stability estimates and proved optimal convergence rates 
(due to the maximal available regularity of the exact solution) in 
a more general setting in physical domain. Then we have specified some 
``goal oriented'' numerical schemes. 
These numerical schemes are  
                derived using a variety 
Galerkin methods such as Standard
                Galerkin, Semi-Streamline Diffusion, Characteristic
                Galerkin and  Characteristic Streamline
                Diffusion methods. Our focus has been in two of these 
approximation schemes: (i) the Semi-Streamline Diffusion and (ii) the 
Characteristic Streamline
                Diffusion methods.\\

For these two setting, we derived a priori error estimates and
formulated the adaptive algorithm. 
 Since in our numerical tests we have used a closed form of the
 analytic solution, therefore it suffices to use a priori error estimates
 for the local mesh refinements.
Numerically we tested our adaptive
algorithm for different type of initial data in (\ref{model}) in three tests
with different mesh refinement parameter $\widetilde{\gamma}$ in the mesh
refinement criterion (\ref{adaptalg}). The goal of our numerical
experiments was to remove oscillatory behavior of the computational
solution as well as removing of the formation of the artificial layer.

 Using Tables and Figures of section \ref{sec:num} we can conclude
 that the oscillatory behavior of the computed solutions which were
 obtained in \cite{MA-EL:2007} are appearing for problems with
 non-smooth initial data and on non-refined meshes. In this work we
 removed these oscillations by  adaptive mesh refinement and
 decreasing the dominance of the coefficient in the convection term.

\section*{Acknowledgments}

The research of first and second authors is supported by the
Swedish Research Council (VR).

\bibliography{references}         

\begin{thebibliography}{99}


\bibitem{adams1975}
	A.\ R.\ Adams,
	\emph{Sobolev spaces},
	Academic press,
	New York,
	1975.





\bibitem{MA97}  M.\ Asadzadeh,
	Streamline diffusion methods for Fermi and Fokker-Planck equations,
	\emph{Transport Theory Statistica Physics},
	26(3), pp. 319--340, 1997.


\bibitem{MA-EL:2007}
	M.\ Asadzadeh  and E.\ Larsen,
	Linear transport equations in flatland with small angular diffusion and their finite element approximation,
	\emph{Mathematical and Computer Modelling, Elsevier},
	pp. 495 -- 514, 2007.


\bibitem{Mohammad2000}
	M.\ Asadzadeh,
	A posteriori error estimates for the Fokker-Planck and Fermi pencil beam equations,
	\emph{Math. Models Meth. Appl Sci.},
	V. 48, 10 (5),
	pp. 737--769,
	2000.




\bibitem{MA-Sopa:2002}
	M.\ Asadzadeh and A.\ Sopasakis,
	On fully discrete schemes for the Fermi pencil-beam equation,
		\emph{Comput. Methods Appl. Mech. Engrg.} 
	191 (1),
	pp. 4641--4659,
	2002.

\bibitem{MA2002}
	M.\ Asadzadeh,
	On the stability of characteristic schemes for the fermi equation,
	\emph{Appl. Comput. Math.},
	1(1),
	pp.158--174,
	2002.



\bibitem{Borgers1996}
	C.\ Borgers and E.\ W.\ Larsen, 
	Asymptotic derivation of the fermi pencil-beam approximation,
	\emph{Nucl. Sci. Engrg},
	123,
	pp. 343--357,
	1996.


\bibitem{Brenner} S.\ C.\ Brenner, L.\ R.\ Scott, 
\emph{The Mathematical theory of finite element methods},
  Springer-Verlag, Berlin, 1994.


\bibitem{Ciarlet:80}
P. G.\ Ciarlet, \emph{The finite element method for elliptic problems},
North-Holland,  Amsterdam, New York, Oxford, 1980.
	 1941.

\bibitem{Eyges:1948}
	L.\ Eyges,
	{Multiple scattering with energy loss},
	\emph{Phys. Rev.},
	 74,
	pp.1534--35,
	 1948.

   




\bibitem{Johnson1992}
 C.\ Johnson, 
A new approach to algorithms for convection problems which are based on 
exact transport + projection,
	\emph{Comput. Methods Appl. Mech. Engrg.},
	100, pp. 45--62,
	1992.

\bibitem{Larsen_etal:85}
 E.\ W.\ Larsen, C.\ D.\ Levermore, G.\ C.\ Pomraning, J.\ G.\ Sanderson,
Discretization methods for one-dimensional Fokker-Planck operators.
\emph{J. Comput. Phys.}  61,  no. 3, pp. 359--390, 1985.


\bibitem{Luo.Brahme:93}
        Z.-M.\ Luo, A.\ Brahme,  
        ``An Overview of the Transport Theory of Charged Particles'', 
        \emph{Raiat. Phys. Chem., {\bf 41}, 673}, (1993). 


   

%

\bibitem{Pomraning:92}
       G.\ C.\ Pomraning,
       The Fokker-Planck Operator as an Asymptotic Limit, 
        \emph{Math. Models Meth., Ap. Sci., {\bf 2}, 21} (1992).
 
\bibitem{Prinja_Pomraning:96}
A.\ K.\ Prinja, G.\ C.\ Pomraning,
One-dimensional beam transport.
\emph{Transport Theory Statist. Phys.}  25,  no. 2, 231--247, 1996.


\bibitem{Rossi1941}
	B.\ Rossi, K.\ Greisen,
	Cosmic-ray theory,
	\emph{Rev. Mod. Phys.},
	 13, ,
	 pp. 309--340,



\end{thebibliography}

\end{document}